\documentclass{amsart}
\pdfoutput=1

\usepackage{amssymb}
\usepackage{microtype}
\usepackage{tikz}
\usetikzlibrary{arrows}
\usepackage{booktabs}
\usepackage[pdftex,colorlinks,citecolor=black,linkcolor=black,urlcolor=black,bookmarks=false]{hyperref}
\def\MR#1{\href{http://www.ams.org/mathscinet-getitem?mr=#1}{MR#1}}
\def\arXiv#1{arXiv:\href{http://arXiv.org/abs/#1}{#1}}
\usepackage{doi}

\newtheorem{theorem}{Theorem}[section]
\newtheorem{proposition}[theorem]{Proposition}
\newtheorem{lemma}[theorem]{Lemma}
\newtheorem{corollary}[theorem]{Corollary}
\newtheorem{conjecture}[theorem]{Conjecture}

\theoremstyle{definition}
\newtheorem{definition}[theorem]{Definition}

\numberwithin{figure}{section}
\numberwithin{equation}{section}
\numberwithin{table}{section}

\newcommand{\N}{\mathbb{N}}

\newcommand{\R}{\mathbb{R}}
\newcommand{\C}{\mathbb{C}}
\newcommand{\E}{\mathbb{E}}
\newcommand{\mC}{\mathcal{C}}
\newcommand{\vol}{\operatorname{vol}}
\newcommand{\arcsec}{\operatorname{arcsec}}
\newcommand{\Ai}{\operatorname{Ai}}
\newcommand{\supp}{\operatorname{supp}}
\newcommand{\contour}{\gamma}

\renewcommand{\Im}{\operatorname{Im}}
\renewcommand{\Re}{\operatorname{Re}}

\title{The Gaussian core model in high dimensions}

\author{Henry Cohn}
\address{Microsoft Research New England\\
One Memorial Drive\\
Cambridge, MA 02142} \email{cohn@microsoft.com}

\author{Matthew de Courcy-Ireland}
\address{Department of Mathematics\\
Princeton University\\
Princeton NJ 08544} \email{mdc4@math.princeton.edu}

\date{February 21, 2018}

\thanks{Matthew de Courcy-Ireland was supported by an internship at Microsoft Research
New England and by an NSERC PGS D grant.}

\begin{document}

\begin{abstract}
We prove lower bounds for energy in the Gaussian core model, in which point
particles interact via a Gaussian potential. Under the potential function
$t \mapsto e^{-\alpha t^2}$ with $0 < \alpha < 4\pi/e$, we show that no
point configuration in $\R^n$ of density $\rho$ can have energy less than
$(\rho+o(1))(\pi/\alpha)^{n/2}$ as $n \to \infty$ with $\alpha$ and $\rho$
fixed. This lower bound asymptotically matches the upper bound of $\rho
(\pi/\alpha)^{n/2}$ obtained as the expectation in the Siegel mean value
theorem, and it is attained by random lattices. The proof is based on the
linear programming bound, and it uses an interpolation construction
analogous to those used for the Beurling-Selberg extremal problem in
analytic number theory.  In the other direction, we prove that the upper
bound of $\rho (\pi/\alpha)^{n/2}$ is no longer asymptotically sharp when
$\alpha > \pi e$.  As a consequence of our results, we obtain bounds in
$\R^n$ for the minimal energy under inverse power laws $t \mapsto
1/t^{n+s}$ with $s>0$, and these bounds are sharp to within a constant
factor as $n \to \infty$ with $s$ fixed.
\end{abstract}

\maketitle

\section{Introduction}
\label{sec:intro}

In the \emph{Gaussian core model} \cite{S}, point particles interact via a
Gaussian potential function $t \mapsto e^{-\alpha t^2}$, and we wish to
minimize the resulting potential energy between them. Specifically, let $\mC$
be a discrete subset of $\R^n$ with \emph{density} $\rho>0$.  In other words,
the number of points of $\mC$ in the closed ball $B_r^n(0)$ of radius $r$
about the origin is asymptotic to $\rho
\vol\mathopen{}\big(B_r^n(0)\big)\mathclose{}$ as $r \to \infty$. To define
the energy of $\mC$ with respect to the potential function $f$, we sum over
pairs of points in a large ball:

\begin{definition} \label{def:lowerenergy}
Let $f \colon (0,\infty) \to \R$ be any function.  The \emph{lower $f${\kern
-1pt}-energy} of a discrete subset $\mC$ of $\R^n$ is 
\[
\liminf_{r \to \infty} \frac{1}{\#\big(\mC \cap B_r^n(0)\big)} \sum_{\substack{x,y \in \mC \cap B_r^n(0)\\x \ne y}} f\big(|x-y|\big).
\]
If this limit exists, and not just the limit inferior, then we call it the
\emph{$f${\kern -1pt}-energy} of $\mC$. 
\end{definition}

The energy exists whenever $\mC$ is periodic and $f$ is sufficiently rapidly
decreasing. For example, if $\mC$ is a lattice $\Lambda$ in $\R^n$ (i.e., a
discrete subgroup of rank $n$), then its energy is
\[
\sum_{x \in \Lambda \setminus\{0\}} f\big(|x|\big).
\]
In other words, it is a weighted sum over the distances that occur between
points in $\Lambda$.  In addition to being of interest in physics and
geometry \cite{C17}, this quantity is number-theoretically meaningful
\cite{SaSt}. Recall that the \emph{theta series} of $\Lambda$ is the modular
form
\[
\Theta_\Lambda(z) = \sum_{x \in \Lambda} e^{\pi i |x|^2 z};
\]
then the $f${\kern -0.5pt}-energy of $\Lambda$ for $f(t) = e^{-\alpha t^2}$ 
is simply $\Theta_\Lambda(\alpha i/\pi)-1$.

Given the dimension $n$, density $\rho$, and potential function $f$, how
should the points of $\mC$ be arranged so as to minimize the energy? We
define the \emph{minimal energy} to be the infimum of the lower energies of
all such configurations.  See \cite{C10,C17} for background on ground state
and energy minimization problems.

Of course one could consider many sorts of potential functions, but Gaussians
play a special role in energy minimization, because they span the cone of
completely monotonic functions of squared Euclidean distance (see Theorem~12b
in \cite[p.~161]{W}).  This cone contains all inverse power laws, because
\begin{equation} \label{eq:Mellin}
\frac{1}{t^s} = \int_0^\infty e^{-\alpha t^2} \frac{\alpha^{s/2-1}}{\Gamma(s/2)} \, d\alpha
\end{equation}
for $t>0$ and $s>0$.  Thus, lower bounds for Gaussian energy automatically
yield corresponding lower bounds for inverse power laws. In number-theoretic
terms, \eqref{eq:Mellin} amounts to the fact that Epstein zeta functions are
Mellin transforms of theta functions.  See \cite{CK07} for further
justification of why this cone is a natural generalization of inverse power
laws.

One can prove the existence of low-energy configurations using the Siegel
mean value theorem \cite{Siegel1945} as follows.  There is a canonical
probability measure on lattices of determinant $1$ (equivalently, density
$1$) in $\R^n$, which is characterized by $\mathrm{SL}_n(\R)$-invariance. For
$n>1$, the Siegel mean
value theorem says that the expected $f${\kern -0.5pt}-energy 
of a random lattice chosen according to this probability measure is
\[
\int_{\R^n} f\big(|x|\big) \, dx.
\]
When $f(t) = e^{-\alpha t^2}$, the expected energy equals
$(\pi/\alpha)^{n/2}$, and thus there exists a lattice of determinant $1$
with energy at most $(\pi/\alpha)^{n/2}$. The most striking case is when
$\alpha=\pi$, in which case the bound becomes $1$.  More generally,
rescaling shows that there is a lattice of density $\rho$ with energy at
most
\[
\rho \int_{\R^n} f\big(|x|\big)\, dx = \rho (\pi/\alpha)^{n/2}.
\]
We call this bound the \emph{expectation bound}.

It is natural to ask how sharp the expectation bound is. For example, are
there configurations in high dimensions with density $1$ and energy $0.99$
when $\alpha=\pi$?  We show in this paper that no such configurations exist.

More generally, we prove a matching lower bound for energy in high
dimensions under sufficiently wide Gaussian potential functions.
Specifically, we will prove asymptotically sharp bounds for energy as $n \to
\infty$ with $\alpha$ and $\rho$ fixed, as long as $\alpha$ is sufficiently
small. Our bound is sharp for $\alpha < 4\pi/e = 4.62290939\dots$:

\begin{theorem} \label{thm:sharp}
When $f(t) = e^{-\alpha t^2}$ with $0 < \alpha < 4\pi/e$, the minimal
$f${\kern -1pt}-energy 
in $\R^n$ for configurations of density $\rho$ is
$(\rho+o(1))(\pi/\alpha)^{n/2}$ as $n \to \infty$ with $\alpha$ and $\rho$
fixed, or more generally with $(\alpha,\rho)$ confined to a compact subset of
$(0,4\pi/e) \times (0,\infty)$.
\end{theorem}

Note that the $\rho+o(1)$ factor is outside the power of $n/2$, which makes
it a much stronger bound than $\rho (\pi/\alpha+o(1))^{n/2}$ would be. In
particular, when $\alpha=\pi$ the minimal energy converges to $\rho$ as $n
\to \infty$.

For $\alpha < 4\pi/e$, this theorem implies that random lattices are
asymptotically optimal for energy.  Specifically, a $1-o(1)$ fraction of all
lattices must have energy $(\rho+o(1))(\pi/\alpha)^{n/2}$, because no
configuration has energy less than $(\rho+o(1))(\pi/\alpha)^{n/2}$ and the
average over all lattices is $\rho(\pi/\alpha)^{n/2}$.  Thus, imposing the
crystalline structure of a lattice does not asymptotically raise the minimal
energy compared with amorphous particle arrangements, and no lattice is
substantially better than a typical lattice. We do not know whether the same
is true when $\alpha \ge 4\pi/e$.

For fixed $n$ and $\rho$, the Gaussian core model degenerates to the sphere
packing problem as $\alpha \to \infty$: the energy of a periodic
configuration is asymptotically determined by the minimal distance between
distinct points, and minimizing energy amounts to maximizing the minimal
distance. See Section~\ref{sec:spherepacking} for a quantitative account of
this relationship. The question of whether lattices are near-optimal sphere
packings in high dimensions is a major unsolved problem. Our theorem can be
viewed as a partial answer, which says that near-optimality of lattices holds
for $\alpha < 4\pi/e$.  We see no reason to expect it to hold for fixed $n$
as $\alpha \to \infty$, i.e., in the case of sphere packing, but perhaps it
holds for each fixed $\alpha$ as $n \to \infty$.

When $\alpha \ge 4\pi/e$, Theorem~\ref{thm:sharp} no longer applies, and the
best lower bound we know how to prove is the following.

\begin{theorem} \label{thm:notsharp}
When $f(t) = e^{-\alpha t^2}$ with $\alpha \ge 4\pi/e$, the minimal $f${\kern
-1pt}-energy 
in $\R^n$ for configurations of density $\rho$ is at least
\[
\rho \left(\frac{1}{2} e^{1-\alpha e /(8\pi)} + o(1)\right)^n
\]
as $n \to \infty$ with $\alpha$ and $\rho$ fixed, or more generally with
$(\alpha,\rho)$ confined to a compact subset of $[4\pi/e,\infty) \times
(0,\infty)$.
\end{theorem}

Note that
\[
\frac{1}{2} e^{1-\alpha e /(8\pi)} = \sqrt{\frac{\pi}{\alpha}}
\]
when $\alpha = 4\pi/e$, which means the exponential rates in
Theorems~\ref{thm:sharp} and~\ref{thm:notsharp} vary continuously as a
function of $\alpha$. We deduce those theorems from the following explicit
bound:

\begin{theorem} \label{thm:main}
Let $f\colon \R \to \R$ be a Gaussian, let $\lambda_1 < \lambda_2 < \cdots$
be the positive roots of the Bessel function $J_{n/2}$, and choose $r$ so
that $\vol\mathopen{}\big(B_{r/2}^n(0)\big)\mathclose{} = \rho$. Then the
minimal $f${\kern -1pt}-energy 
for point configurations of density $\rho$ in $\R^n$ is at least
\[
\frac{n}{2^{n-1}(n/2)!^2}\sum_{m=1}^\infty \frac{\lambda_m^{n-2}}{J_{n/2-1}(\lambda_m)^2} f\mathopen{}\left(\frac{\lambda_m}{\pi r}\right)\mathclose{}.
\]
\end{theorem}

Here $(n/2)!$ means $\Gamma(n/2+1)$ when $n$ is odd, and
\[
\vol\mathopen{}\big(B_{r/2}^n(0)\big)\mathclose{} = \frac{\pi^{n/2} r^n}{(n/2)!2^n}.
\]

Our proof of Theorem~\ref{thm:main} is based on the linear programming bound
from Section~9 of \cite{CK07}, which we prove in slightly greater generality
in Proposition~\ref{prop:LP}. This lower bound depends on the choice of an
auxiliary function, and the best choice is not known in general. We obtain
Theorem~\ref{thm:main} by making a specific choice, which is surely
suboptimal in general yet still performs well.  Our auxiliary function is
based on interpolation at Bessel function roots, and it yields the analogue
for energy minimization of Proposition~6.1 from \cite{CE03}.

Theorem~\ref{thm:main} draws no distinction between the cases where $\alpha <
4\pi/e$ and those where $\alpha \ge 4\pi/e$.  Instead, the phase transition
between these cases arises in the asymptotic analysis of the bound.  We do
not know whether the phase transition at $\alpha = 4\pi/e$ reflects an actual
change in the behavior of energy minimization at that point, but we can prove
that the expectation bound is not sharp when $\alpha$ is large. Specifically,
in Proposition~\ref{prop:condexp} we obtain an energy upper bound of
\[
\rho \left(\frac{\pi}{\alpha}\right)^{n/2} \left( e^{1/2 - \alpha/(2\pi e)}+o(1) \right)^n,
\]
which is an exponential improvement on the expectation bound when $\alpha >
\pi e$.  Thus, Theorem~\ref{thm:sharp} covers most of the range of values of
$\alpha$ for which the expectation bound could be sharp.  We do not know what
happens between $4\pi/e$ and $\pi e$, or whether further improvements are
possible when $\alpha > \pi e$.

Because of the uniformity in Theorems~\ref{thm:sharp} and~\ref{thm:notsharp},
these theorems can be applied to analyze other potential functions, as long
as they are completely monotonic functions of squared distance.  In
particular, we can use \eqref{eq:Mellin} to analyze inverse power laws.  If
$f(t) = 1/t^{n+s}$ with $s>0$, then the expectation bound is infinite but
Lemma~\ref{lem:condexp} (with $r$ chosen so that
$\vol\mathopen{}\big(B_r^n(0)\big)\mathclose{} = 1/n$) implies that for
$\rho$ and $s$ fixed, there exist configurations of density $\rho$ and
$f${\kern -0.5pt}-energy at most 
\begin{equation} \label{eq:invpowbd}
\left(\frac{2\rho}{s}+o(1)\right) \frac{\pi^{(n+s)/2} e^{s/2}}{\Gamma\big((n+s)/2\big)}
\end{equation}
as $n \to \infty$.  Combining Theorem~\ref{thm:sharp} with the inequality
\[
\frac{1}{t^{n+s}} \ge \int_0^{4\pi/e} e^{-\alpha t^2} \frac{\alpha^{(n+s)/2-1}}{\Gamma\big((n+s)/2\big)} \, d\alpha,
\]
which follows from \eqref{eq:Mellin}, proves a lower bound equal to
$(2/e)^s+o(1)$ times the upper bound \eqref{eq:invpowbd}, which is therefore
sharp to within a constant factor as $n \to \infty$.  If the range of
validity of Theorem~\ref{thm:sharp} could be extended from $4\pi/e$ to $\pi
e$, then the resulting lower bound for inverse power laws would be
asymptotically sharp (i.e., the constant factors would match).

Table~\ref{table:numbers} compares our bound from Theorem~\ref{thm:main} with
other bounds when $\alpha = \pi$ and $\rho=1$.  The first column shows our
bound, the second shows the result of numerically optimizing the auxiliary
function in the linear programming bound, and the third shows the lowest
energy currently known. The table illustrates the convergence of our bound to
$\rho$ as $n \to \infty$ when $\alpha = \pi$, and it provides evidence that
in this case our bound is not much worse than the full linear programming
bound.

Aside from the high-dimensional behavior, the most striking aspect of the
table is the seemingly matching lower and upper bounds when $n=1$, $2$, $8$,
or $24$. For $n=1$, Proposition~9.6 from \cite{CK07} proves a sharp bound, as
does Theorem~\ref{thm:main} in the present paper. When $n=2$, $8$, or $24$,
the agreement between the bounds is a special case of Conjecture~9.4 from
\cite{CK07}, which is highly plausible given recent results on sphere packing
\cite{V,CKMRV}.

\begin{table}
\caption{A comparison of our lower bound from Theorem~\ref{thm:main} for energy in $\R^n$ when
$\alpha=\pi$ and $\rho=1$, the numerically optimized linear programming bound, and the
lowest energy currently known (from \cite{CKS} or the expectation bound).} \label{table:numbers}
\begin{tabular}{rlll}
\toprule
$n$ & Our bound & LP bound & Current record\\
\midrule
$1$ & $0.08643481\dots$ & $0.08643481\dots$ & $0.08643481\dots$\\
$2$ & $0.15702654\dots$ & $0.15959526\dots$ & $0.15959526\dots$\\
$3$ & $0.21736068\dots$ & $0.22321782\dots$ & $0.23153532\dots$\\
$4$ & $0.27028747\dots$ & $0.27956960\dots$ & $0.28576449\dots$\\
$5$ & $0.31750042\dots$ & $0.33011740\dots$ & $0.34868410\dots$\\
$6$ & $0.36010894\dots$ & $0.37587226\dots$ & $0.38874675\dots$\\
$7$ & $0.39889096\dots$ & $0.41756856\dots$ & $0.42445404\dots$\\
$8$ & $0.43442005\dots$ & $0.45576289\dots$ & $0.45576289\dots$\\
$9$ & $0.46713560\dots$ & $0.49089167\dots$ & $0.49771252\dots$\\
$24$ & $0.76270306\dots$ & $0.79965280\dots$ & $0.79965280\dots$\\
$100$ & $0.99321117\dots$ & $0.99735690\dots$ & $1$\\
$200$ & $0.99991895\dots$ & $0.99998973\dots$ & $1$\\
$500$ & $0.99999999\dots$ & $0.99999999\dots$ & $1$\\
\bottomrule
\end{tabular}
\end{table}

Theorem~\ref{thm:main} is analogous to the results of \cite{BDHSS} for
spherical codes, and as in that paper we prove a corresponding optimality
result under certain conditions (Proposition~\ref{prop:opt}). Our
construction of the auxiliary function for the linear programming bound is
also analogous to the solution of the Beurling-Selberg extremal problem in
analytic number theory \cite{CL, CLV, CV1, CV2, GV, HV, Va}, which arose
independently in Selberg's work on sieve theory and Beurling's work on
analytic functions (see \cite[p.~226]{Selberg} for historical comments).  In
particular, the relevant case for our work is Gaussian subordination
\cite{CLV,CL}.  In that problem, we are given parameters $n$, $\alpha$, and
$r$, and the goal is to find the entire function $h$ of exponential type at
most $2\pi r$ such that $h$ maps $\R$ to $\R$, $h(t) \le e^{-\alpha t^2}$ for
all $t \in \R$, and
\[
\int_{\R^n} h\big(|x|\big) \, dx
\]
is maximized.  This goal is achieved in \cite{CLV,CL}.  Optimizing the
auxiliary function in the linear programming bound involves somewhat
different constraints, but we construct our auxiliary function using a
similar approach based on interpolation, and our proof techniques can be used
to give a new proof of Theorems~2 and~3 from \cite{CLV} and their
higher-dimensional analogues from \cite{CL}. Our results are also similar to
Poltyrev's work \cite{P} on communicating over a channel with Gaussian random
noise, in which he determined the exact error exponents for a certain
parameter range and showed that they are achieved by random lattices.

In the remainder of this paper, we first develop the linear programming bound
in Section~\ref{sec:LP}.  In Section~\ref{sec:PDK} we prove a key lemma about
positive-definite functions, and in Section~\ref{sec:exp} we recall some
background about entire functions of exponential type and formulate a
strategy for proving Theorem~\ref{thm:main}. We then apply these results to
prove Theorem~\ref{thm:main} via Hermite interpolation in
Sections~\ref{sec:interp} and~\ref{sec:entire}. In Section~\ref{sec:asymp},
we carry out an asymptotic analysis of this bound to deduce
Theorems~\ref{thm:sharp} and~\ref{thm:notsharp}. We derive an improved upper
bound for energy in Section~\ref{sec:condexp}. Finally, we draw a more
concrete connection between sphere packing and the Gaussian core model in
Section~\ref{sec:spherepacking}, and we conclude with open problems in
Section~\ref{sec:open}.

\section*{Acknowledgments}

We thank Ganesh Ajjanagadde, Peter Sarnak, and the anonymous referees for
their helpful comments on the manuscript.

\section{The linear programming bound} \label{sec:LP}

In this section we develop the linear programming bound for energy
minimization.  This bound is essentially Proposition~9.3 of \cite{CK07}, but
the proof given there works only for periodic configurations.  Here we
extend it to arbitrary configurations.

\begin{definition}
A continuous function $h \colon \R^n \to \C$ is \emph{positive definite} if
$h(-x) = \overline{h(x)}$ for all $x \in \R^n$, and for all $N \in \N$ and
$x_1,\dots,x_N \in \R^n$, the $N \times N$ Hermitian matrix
\[
\big(h(x_j-x_k)\big)_{1 \le j,k \le N}
\]
is positive semidefinite.
\end{definition}

The latter condition is equivalent to asserting that for all $x_1,\dots,x_N
\in \R^n$ and $t_1,\dots,t_N \in \C$,
\begin{equation}
\label{eq:posdefreform}
\sum_{1 \le j,k \le N} t_j \overline{t_k} h(x_j-x_k) \ge 0.
\end{equation}

We will need several properties of positive-definite functions. They are
closed under multiplication, by the Schur product theorem (Theorem~7.5.3 in
\cite{HJ}), which says that positive-semidefinite matrices are closed under
the Hadamard product.  Furthermore, if $h$ is positive definite, then the $2
\times 2$ matrix
\[
\begin{bmatrix}
h(0) & h(x)\\
h(-x) & h(0)
\end{bmatrix}
\]
is positive semidefinite for every $x \in \R^n$, from which it follows that
$h(0) \ge 0$ and $|h(x)| \le h(0)$ by taking the determinant.

Recall that Bochner's theorem characterizes positive-definite functions as
those of the form
\[
x \mapsto \int_{\R^n} e^{2\pi i \langle x,y \rangle} \, d\mu(y),
\]
where $\mu$ is a finite measure on $\R^n$ with respect to the Borel
$\sigma$-algebra (see Theorem~6.6.6 in \cite{Simon}).  Define the
\emph{Fourier transform} $\widehat{h}$ of an integrable function $h \colon
\R^n \to \C$ by
\[
\widehat{h}(y) = \int_{\R^n} h(x) e^{-2\pi i \langle x,y\rangle} \, dx.
\]
Then it follows from Bochner's theorem and Fourier inversion that if $h$ and
$\widehat{h}$ are both integrable, then $h$ is positive definite if and only
if $\widehat{h} \ge 0$.

The linear programming bound can be stated as follows.  As mentioned above,
it is essentially Proposition~9.3 from \cite{CK07}, but we state it here in
slightly greater generality and give a different proof.

\begin{proposition}[Cohn and Kumar \cite{CK07}] \label{prop:LP}
Let $f \colon (0,\infty) \to \R$ be any function, and suppose $h \colon \R^n
\to \R$ is continuous, positive definite, and integrable.  If $h(x) \le
f\big(|x|\big)$ for all $x \in \R^n \setminus \{0\}$, then every subset of
$\R^n$ with density $\rho$ has lower $f${\kern -1pt}-energy at least $\rho 
\widehat{h}(0) - h(0)$.
\end{proposition}

The proof will follow the approach used to prove Theorem~3.3 in \cite{CZ14}.

\begin{proof}
Let $\mC$ be a subset of $\R^n$ of density $\rho$.  For each $r>0$, let
\[
\mC_r = \{x \in \mC : |x| \le r\},
\]
let $N_r = \# \mC_r$, and let $V_r =
\vol\mathopen{}\big(B_r^n(0)\big)\mathclose{}$. Then
\[
\lim_{r \to \infty} \frac{N_r}{V_r} = \rho.
\]
To avoid dividing by zero when computing energy, we will restrict our
attention to $r$ such that $N_r>0$.

The proof will be based on renormalizing a sum over $\mC$ by subtracting a
uniform background density.  Specifically, we will consider the signed
measure
\[
\nu = \sum_{x \in \mC_r} \delta_x - \frac{N_r}{V_r} \mu_{R},
\]
where $\delta_x$ denotes a delta function at $x$, $\mu_R$ denotes Lebesgue
measure on the ball of radius $R$ centered at the origin, and $R =
r+\sqrt{r}$.  The shift by $\sqrt{r}$ simplifies one of the limits we need
below, but it is not conceptually important.  Ignoring this shift, observe
that
\[
\sum_{x \in \mC_r} \delta_x - \frac{N_r}{V_r} \mu_{r}
\]
has integral zero, which explains the factor of $N_r/V_r$.

Because $h$ is positive definite, it follows from approximating integrals
with respect to $\mu_R$ by Riemann sums and applying \eqref{eq:posdefreform}
that
\[
\iint h(x-y) \, d\nu(x) \, d\nu(y) \ge 0.
\]
Equivalently,
\[
\frac{N_r^2}{V_r^2}  \iint_{|x|,|y| \le R} h(x-y) \, dx \, dy - \frac{2N_r}{V_r}  \sum_{x \in \mC_r} \int_{|y| \le R} h(x-y) \, dy
+ \sum_{x,y \in \mC_r} h(x-y) \ge 0.
\]
Applying the inequality $h(x-y) \le f\big(|x-y|\big)$ and rearranging yields
\begin{equation} \label{eq:fundineq}
\begin{split} \frac{1}{N_r} \sum_{\substack{x,y \in \mC_r\\ x \ne y}}
f\big(|x-y|\big) &\ge \frac{2N_r}{V_r} \cdot \frac{1}{N_r} \sum_{x \in \mC_r}
\int_{|y| \le R} h(x-y) \, dy
- h(0)\\
& \quad \phantom{} - \frac{N_r}{V_r} \cdot \frac{1}{V_r} \iint_{|x|,|y| \le R} h(x-y) \, dx \, dy.
\end{split}
\end{equation}
To complete the proof, we will show that the right side of
\eqref{eq:fundineq} converges to $\rho \widehat{h}(0) - h(0)$ as $r \to
\infty$.  First, we claim that
\[
\frac{1}{N_r} \sum_{x \in \mC_r} \int_{|y| \le R} h(x-y) \, dy \to \widehat{h}(0),
\]
for the following reason. Each fixed summand converges to $\widehat{h}(0)$,
because
\[
\int_{|y| \le R} h(x-y) \, dy \to \int_{\R^n} h(x-y) \, dy = \widehat{h}(0),
\]
and all we need to verify is that this limit holds uniformly for $x \in
\mC_r$. To check the uniformity, we note that $\{x-y : |y| \le R\}$ contains
$B^n_{R-|x|}(0)$.  In particular, because $|x| \le r$ and $R-r = \sqrt{r} \to
\infty$, the integrals
\[
\int_{|y| \le R} h(x-y) \, dy
\]
include all values of $x-y$ of length at most $\sqrt{r}$. Thus, they converge
uniformly to $\widehat{h}(0)$. Similarly,
\[
\frac{1}{V_r} \iint_{|x|,|y| \le R} h(x-y) \, dx \, dy \to \widehat{h}(0),
\]
since for $|x| \le r$ the $y$-integrals converge uniformly to $\widehat{h}(0)$,
and the contributions from $r \le |x| \le r+\sqrt{r}$ are negligible compared
with $V_r$. Combining these limits with $N_r/V_r \to \rho$ and
\eqref{eq:fundineq} yields
\[
\liminf_{r \to \infty} \frac{1}{N_r} \sum_{\substack{x,y \in \mC_r\\ x \ne y}} f\big(|x-y|\big) \ge \rho \widehat{h}(0)
- h(0),
\]
as desired.
\end{proof}

\section{Positive-definite functions} \label{sec:PDK}

The proof of Theorem~\ref{thm:main} will make crucial use of the following
proposition, which says that removing successive roots from a suitable Bessel
function always yields a positive-definite function.

\begin{proposition} \label{prop:peel}
Let $\nu = n/2-1$, and let $\lambda_1 < \lambda_2 < \cdots$ be the positive
roots of $J_\nu$.  Then for each $k \ge 0$, the function
\[
x \mapsto \frac{J_\nu\big(|x|\big)}{|x|^\nu \left(1-\frac{|x|^2}{\lambda_1^2}\right) \dots \left(1-\frac{|x|^2}{\lambda_k^2}\right)}
\]
is positive definite on $\R^n$ (where of course we define the function by
continuity when its denominator vanishes).
\end{proposition}

This assertion is well known for $k=0$, and it is not hard to prove it for
$k=1$ using the Christoffel-Darboux formula. For general $k$, it is a Bessel
function analogue of Theorem~3.1 from \cite{CK07}:

\begin{theorem}[Cohn and Kumar \cite{CK07}] \label{theorem:CK07}
Let $p_0,\dots,p_n$ be the monic orthogonal polynomials with respect to some
measure on $\R$, where $\deg p_i = i$, let $\alpha \in \R$, and let
\[
r_1 < r_2 < \dots < r_n
\]
be the roots of $p_n+\alpha p_{n-1}$. Then for $1 \le k<n$, the polynomial
\[
\prod_{i=1}^k (t-r_k)
\]
has positive coefficients in terms of $p_0(t),\dots,p_k(t)$.
\end{theorem}

Proposition~\ref{prop:peel} could very likely be proved by adapting the proof
given in \cite{CK07}, but we will deduce it from this theorem by taking a
suitable limit of polynomials.

The polynomials we will use are those that are positive definite on the unit
sphere $S^{n-1}$ in $\R^n$.  Recall that a function $p \colon [-1,1] \to \R$
is \emph{positive definite on $S^{n-1}$} if for all $N \in \N$ and
$x_1,\dots,x_N \in S^{n-1}$, the $N \times N$ matrix
\[
\Big( p\big(\langle x_i,x_j \rangle\big)\Big)_{1 \le i,j \le N}
\]
is positive semidefinite.\footnote{Strictly speaking, it is the function
$\widetilde{p} \colon S^{n-1} \times S^{n-1} \to \R$ defined by
$\widetilde{p}(x,y) = p(\langle x,y\rangle)$ that is positive definite, but
it is more convenient to talk about $p$.} Schoenberg's theorem \cite{Sch42}
characterizes positive-definite functions on $S^{n-1}$ as the nonnegative
linear combinations of the Jacobi polynomials $P^{((n-3)/2,(n-3)/2)}_m$ with
$m \ge 0$. Jacobi polynomials with these parameters are also known as
Gegenbauer polynomials or ultraspherical polynomials.

We will deduce Proposition~\ref{prop:peel} from Theorem~\ref{theorem:CK07}
applied to the Jacobi polynomials $P^{(\nu,\nu)}_m$.  Note that the
parameter shift from $(n-3)/2$ to $\nu=n/2-1$ means that these polynomials
are positive definite on $S^n$, not just $S^{n-1}$.

\begin{proof}[Proof of Proposition~\ref{prop:peel}]
What we must prove is that for all $x_1,\dots,x_N \in \R^n$, the $N \times N$
matrix whose $i,j$ entry is
\[
\frac{J_\nu\big(|x_i-x_j|\big)}{|x_i-x_j|^\nu \left(1-\frac{|x_i-x_j|^2}{\lambda_1^2}\right) \dots \left(1-\frac{|x_i-x_j|^2}{\lambda_k^2}\right)}
\]
is positive semidefinite.

To reduce from Bessel functions to Jacobi polynomials, we will use the
Mehler-Heine formula for Jacobi polynomials (Theorem~8.1.1 in \cite{S75}),
which says that
\begin{equation}
\label{eq:mehler-heine}
\lim_{m \to \infty}
(2m)^{-\nu} P_m^{(\nu,\nu)}\left(1-\frac{z^2}{2m^2}\right) = z^{-\nu}
J_\nu(z),
\end{equation}
uniformly for $z$ in any compact subset of $\C$.  Furthermore, Hurwitz's
theorem implies that if $r_{1,m} > \dots > r_{m,m}$ are the roots of
$P_m^{(\nu,\nu)}$, then
\begin{equation}
\label{eq:hurwitz}
r_{i,m} = 1 - \frac{\lambda_i^2}{2m^2} + o(1/m^2)
\end{equation}
for each fixed $i$ as $m \to \infty$ (Theorem~8.1.2 in \cite{S75}).

To take advantage of the fact that the polynomials $P_m^{(\nu,\nu)}$ are
positive definite on the unit sphere $S^n$, we will map $\R^n$ to $S^n$. We
view $\R^n$ as a hyperplane in $\R^{n+1}$.  Let $e_n$ be a unit vector in
$\R^{n+1}$ orthogonal to $\R^n$, and for each positive integer $m$ define a
function $f_m \colon \R^n \to S^n$ by
\[
f_m(x) = \frac{e_n + x/m}{\sqrt{1+|x|^2/m^2}}.
\]
Then a short calculation shows that
\begin{equation}
\label{eq:fm}
\langle f_m(x),f_m(y) \rangle = 1 - \frac{|x-y|^2}{2m^2} + O(1/m^4),
\end{equation}
where the big-$O$ term is uniform when $x$ and $y$ are confined to a compact
set.  It now follows that for all $x,y \in \R^n$,
\[
\lim_{m \to \infty} (2m)^{-\nu} P_m^{(\nu,\nu)}\left(\langle f_m(x),f_m(y) \rangle\right) = |x-y|^{-\nu} J_\nu\big(|x-y|\big).
\]
In particular, the $O(1/m^4)$ error term is handled by the uniformity of
convergence in the Mehler-Heine formula.

By Theorem~\ref{theorem:CK07}, the polynomial
\[
\frac{P_m^{(\nu,\nu)}(z)}{(z-r_{1,m})\dots(z-r_{k,m})}
\]
is a linear combination of $P_0^{(\nu,\nu)}(z),\dots,P_{m-k}^{(\nu,\nu)}(z)$
with nonnegative coefficients for each $k \le m$.  It follows that this
polynomial is positive definite on $S^n$.  Thus, for all $x_1,\dots,x_N \in
\R^n$, the $N \times N$ matrix with $i,j$ entry
\[
\frac{P_m^{(\nu,\nu)}(\langle f_m(x_i),f_m(x_j) \rangle)}{\big(\langle f_m(x_i),f_m(x_j) \rangle-r_{1,m}\big)\dots\big(\langle f_m(x_i),f_m(x_j) \rangle-r_{k,m}\big)}
\]
is positive semidefinite.

It follows from \eqref{eq:hurwitz} and \eqref{eq:fm} that for $1 \le \ell \le
k$,
\[
\lim_{m \to \infty} \frac{2m^2}{\lambda_\ell^2}\big(\langle f_m(x_i),f_m(x_j) \rangle-r_{\ell,m}\big) = 1 - \frac{|x_i-x_j|^2}{\lambda_\ell^2}.
\]
Combining this limit with the Mehler-Heine formula \eqref{eq:mehler-heine},
we find that
\[
\lim_{m \to \infty} \left(\frac{2^{k-\nu} m^{2k-\nu}}{\lambda_1^2 \dots \lambda_k^2} \cdot
\frac{P_m^{(\nu,\nu)}(\langle f_m(x_i),f_m(x_j) \rangle)}{\big(\langle f_m(x_i),f_m(x_j) \rangle-r_{1,m}\big)\dots\big(\langle f_m(x_i),f_m(x_j) \rangle-r_{k,m}\big)} \right)
\]
is equal to
\begin{equation} \label{eq:Jentry}
\frac{J_\nu\big(|x_i-x_j|\big)}{|x_i-x_j|^\nu \left(1-\frac{|x_i-x_j|^2}{\lambda_1^2}\right) \dots \left(1-\frac{|x_i-x_j|^2}{\lambda_k^2}\right)},
\end{equation}
as long as $|x_i-x_j|$ is not among $\lambda_1,\dots,\lambda_k$.  Applying
this limit to the matrix from the previous paragraph shows that the $N \times
N$ matrix with entries \eqref{eq:Jentry} is positive semidefinite, again
assuming $|x_i-x_j|$ is never among $\lambda_1,\dots,\lambda_k$.

All that remains is to deal with the case in which $|x_i-x_j| = \lambda_\ell$
for some $i$, $j$, and $\ell$.  However, that case follows easily by
continuity: \eqref{eq:Jentry} is a continuous function of $|x_i-x_j|$, and
every configuration $x_1,\dots,x_N$ has arbitrarily small perturbations in
which $\lambda_1,\dots,\lambda_k$ do not occur as distances.
\end{proof}

\section{Entire functions of exponential type} \label{sec:exp}

Recall that one version of the Paley-Wiener theorem characterizes Fourier
transforms of compactly supported distributions as entire functions of
exponential type:

\begin{proposition} \label{prop:pw}
Let $h$ be a function from $[0,\infty)$ to $\R$.  Then the radial function
$x \mapsto h\big(|x|\big)$ on $\R^n$ is the Fourier transform of a
distribution supported on $B^n_r(0)$ if and only if $h$ extends to an even
entire function on $\C$ for which there are constants $C$ and $k$ such that
\[
|h(z)| \le C (1+|z|)^k e^{2\pi r |{\Im z}|}
\]
for all $z \in \C$.
\end{proposition}

In particular, $h$ is an entire function of exponential type at most $2\pi
r$. For a proof of Proposition~\ref{prop:pw}, see Theorem~7.3.1 in \cite{H},
together with Lemma~3.4 in \cite{Co} for information on how to obtain this
result as a special case of the theorem from \cite{H}.

\begin{corollary} \label{cor:peel}
Let $\nu = n/2-1$, and let $\lambda_1 < \lambda_2 < \cdots$ be the positive
roots of $J_\nu$.  Then for each $k$, the function
\[
x \mapsto \frac{J_\nu\big(|x|\big)}{|x|^\nu \left(1-\frac{|x|^2}{\lambda_1^2}\right) \dots \left(1-\frac{|x|^2}{\lambda_k^2}\right)}
\]
on $\R^n$ has Fourier transform supported in $B_{1/(2\pi)}^{n}(0)$.
\end{corollary}

Note that these functions are the same as those in
Proposition~\ref{prop:peel}.  When $k$ is small the functions are not
integrable, in which case we take their Fourier transforms as tempered
distributions.

\begin{proof}
For $k=0$, the function is $x \mapsto J_\nu\big(|x|\big)/|x|^\nu$, and up to
scaling it is the Fourier transform of a delta function supported on the
sphere of radius $1/(2\pi)$ about the origin, by the Bessel function formula
for the radial Fourier transform (Theorem~9.10.3 in \cite{AAR}).  For $k \ge
1$, the additional factors in the denominator do not disrupt the growth
condition from Proposition~\ref{prop:pw}.
\end{proof}

We will need the following summation formula for entire functions of
exponential type:

\begin{lemma}[Ben Ghanem and Frappier \cite{BGF}] \label{lemma:BGF}
Let $h$ be an even entire function of exponential type at most $2\pi r$, and
suppose that $h(x) = O\big((1+|x|)^{-n-\delta}\big)$ for $x \in \R$ with
$\delta > 0$. If $\lambda_1 < \lambda_2 < \cdots$ are the positive roots of
the Bessel function $J_{n/2}$, then
\[
\frac{n}{2^{n-1}(n/2)!^2}\sum_{m=1}^\infty \frac{\lambda_m^{n-2}}{J_{n/2-1}(\lambda_m)^2} h\mathopen{}\left(\frac{\lambda_m}{\pi r}\right)\mathclose{}
= \vol\mathopen{}\big(B_{r/2}^n(0)\big)\mathclose{} \int_{\R^n} h\big(|x|\big) \, dx - h(0).
\]
\end{lemma}

This lemma is the special case of Theorem~2 in \cite{BGF} with $p=0$ and
$\alpha=n/2$ (see also Theorem~3.1 in \cite{Co}). Note that the decay
exponent $n + \delta > n$ is enough to ensure absolute convergence of the
infinite sum. Indeed, the denominator satisfies $J_{n/2-1}(\lambda_m)^2 \sim
2/(\pi \lambda_m)$, because $J_\nu(x)^2 + J_{\nu+1}(x)^2 \sim 2/(\pi x)$ as
$x \to \infty$ (see \S7.21 of \cite[p.~200]{Watson}). Since $\lambda_m$ grows
linearly with $m$, the summands are dominated by $m^{-1 - \delta}$, which is
summable.

We can now formulate our strategy for the proof of Theorem~\ref{thm:main}.
Let $f\colon \R \to \R$ be a Gaussian, and choose $\lambda_1 < \lambda_2 <
\cdots$ and $r$ as in Theorem~\ref{thm:main}. We will construct a radial
auxiliary function $h \colon \R^n \to \R$ such that
\begin{enumerate}
\item\label{cond:pdk} the function $h$ is continuous, integrable, and
    positive definite,

\item\label{cond:hlef} $h(x) \le f\big(|x|\big)$ for all $x \in \R^n$,

\item\label{cond:bigO} $h(x) = O\big((1+|x|)^{-(n+1)}\big)$,

\item\label{cond:supp} the support of $\widehat{h}$ is contained in
    $B^n_r(0)$, and

\item\label{cond:eq} the functions $f$ and $h$ agree at radius
    $\lambda_m/(\pi r)$ for each $m \ge 1$.
\end{enumerate}
We write $h\big(\lambda_m/(\pi r)\big)$ to denote the common value of $h$ at
radius $\lambda_m/(\pi r)$, so that \eqref{cond:eq} can be rephrased as
\[
h\mathopen{}\left(\frac{\lambda_m}{\pi r}\right)\mathclose{} = f\mathopen{}\left(\frac{\lambda_m}{\pi r}\right)\mathclose{}.
\]

If we use this auxiliary function $h$, then the linear programming bound
tells us that the lower $f${\kern -0.5pt}-energy for configurations of 
density~$\rho$ in $\R^n$ is at least $\rho \widehat{h}(0)-h(0)$.  Given our
choice of $r$, Lemma~\ref{lemma:BGF} implies that
\begin{align*}
\rho\widehat{h}(0)-h(0) &= \frac{n}{2^{n-1}(n/2)!^2}\sum_{m=1}^\infty \frac{\lambda_m^{n-2}}{J_{n/2-1}(\lambda_m)^2} h\mathopen{}\left(\frac{\lambda_m}{\pi r}\right)\mathclose{}\\
&= \frac{n}{2^{n-1}(n/2)!^2}\sum_{m=1}^\infty \frac{\lambda_m^{n-2}}{J_{n/2-1}(\lambda_m)^2} f\mathopen{}\left(\frac{\lambda_m}{\pi r}\right)\mathclose{},
\end{align*}
as desired.  All that remains is to construct $h$.

\section{Interpolation by polynomials} \label{sec:interp}

To complete the proof of Theorem~\ref{thm:main}, we must construct an
auxiliary function $h$ with properties \eqref{cond:pdk} through
\eqref{cond:eq} from the previous section.  In this section, we lay the
groundwork by studying polynomial interpolation.

Recall that in \emph{Hermite interpolation}, we are given distinct points
$t_1,\dots,t_N \in \R$ and multiplicities $k_1,\dots,k_N \in \N$, and we wish
to construct the unique polynomial $p$ of degree less than $k_1+\dots+k_N$
with given values $p^{(k)}(t_j)$ for $1 \le j \le N$ and $0 \le k < k_j$.  In
other words, we specify $p$ to order $k_j$ at each point $t_j$.  By the
Hermite interpolation of a function $f$ we mean that $f$ is used to specify
the values of $p$, i.e., $p^{(k)}(t_j) = f^{(k)}(t_j)$. See Section~2.1 of
\cite{CK07} for a brief review of Hermite interpolation. We will typically
indicate the multiplicities $k_1,\dots,k_N$ by specifying a multiset of
interpolation points
\[
\overbrace{t_1,\dots,t_1}^{k_1},\overbrace{t_2,\dots,t_2}^{k_2},\dots,\overbrace{t_N,\dots,t_N}^{k_N},
\]
with multiplicities indicated by repetition.

We will obtain our auxiliary function $h$ as a limit of Hermite interpolation
polynomials.  It will prove convenient to apply a quadratic change of
variables $z=-t^2$, so that the Gaussian $e^{-\alpha t^2}$ becomes the
exponential function $e^{\alpha z}$.

Fix $\alpha>0$, and define $\lambda_1 < \lambda_2 < \cdots$ and $r$ as in
Theorem~\ref{thm:main}.  We define the sequence of interpolation points
$u_1,u_2,\dots$ by
\[
u_{2j-1} = u_{2j} = -\left(\frac{\lambda_j}{\pi r}\right)^2
\]
for $j \ge 1$.  Aside from the quadratic change of variables, these are
exactly the points at which we wish the auxiliary function to equal the
Gaussian, and the repetition indicates that we will interpolate to second
order.  Note that $|u_j| \asymp j^2$ as $j \to \infty$, in the sense that
there are positive constants $c$ and $C$ (depending on $n$ but not $j$) such
that $cj^2 < |u_j| < Cj^2$ for all $j$.

Let $p_{M}$ be the Hermite interpolation of the function $z \mapsto e^{\alpha
z}$ at $u_1,\dots,u_{M}$. We will define the auxiliary function $h$ by
\begin{equation} \label{eq:hdef}
h(w) = \lim_{M \to \infty} p_{M}\big({-|w|^2}\big)
\end{equation}
for $w \in \R^n$, once we show that this limit exists.

The following simple algebraic lemma will play an important role in the
proof. In the lemma, we introduce a new variable $u_0$.  It will play a
similar role to that of $u_1,u_2,\dots$ (as the notation suggests), but we
can choose its value arbitrarily.

\begin{lemma} \label{lemma:alg}
If $u_0$ and $z$ are complex numbers such that $z \not\in \{u_0,\dots,u_M\}$,
then
\[
\frac{1}{z-u_0} = \sum_{k=0}^M \frac{1}{z-u_k} \prod_{j=k+1}^M \frac{1-u_0/u_j}{1-z/u_j}.
\]
If $z \not\in \{u_0,u_1,\dots\}$, then
\[
\frac{1}{z-u_0} = \sum_{k=0}^\infty \frac{1}{z-u_k} \prod_{j=k+1}^\infty \frac{1-u_0/u_j}{1-z/u_j}.
\]
\end{lemma}

The only dependence of this lemma on our choice of $u_1,u_2,\dots$ is through
their quadratic growth, which will justify taking the limit as $M \to
\infty$.

\begin{proof}
The first identity is trivial for $M=0$, and we will prove it by induction on
$M$. After clearing denominators, it amounts to
\[
\prod_{j=1}^M (z-u_j) = \sum_{k=0}^M \prod_{j=0}^{k-1} (z-u_j) \prod_{j=k+1}^M (u_0-u_j).
\]
If we separate out the terms with $k=M$ and $j=M$ and apply the induction
hypothesis, we find that the right side is
\[
(z-u_1)\dots(z-u_{M-1})(u_0-u_M) + (z-u_0)(z-u_1)\dots(z-u_{M-1}),
\]
which simplifies as desired.

To prove the second identity, we take the limit as $M \to \infty$.  Because
$|u_j|$ grows quadratically as a function of $j$, the sum
\[
\sum_{k=1}^\infty \frac{1}{|z-u_k|}
\]
is finite, and the products
\[
\prod_{j=k+1}^\infty \frac{1-u_0/u_j}{1-z/u_j}
\]
converge and are thus uniformly bounded as a function of $k$.  It follows
that
\[
\sum_{k=0}^M \frac{1}{z-u_k} \prod_{j=k+1}^M \frac{1-u_0/u_j}{1-z/u_j} -
\sum_{k=0}^M \frac{1}{z-u_k} \prod_{j=k+1}^\infty \frac{1-u_0/u_j}{1-z/u_j}
\]
equals
\[
\sum_{k=0}^M \frac{1}{z-u_k} \left(\prod_{j=k+1}^M \frac{1-u_0/u_j}{1-z/u_j}\right)\left(1-\prod_{j=M+1}^\infty \frac{1-u_0/u_j}{1-z/u_j}\right),
\]
which converges to $0$ as $M \to \infty$ because the first factor in
parentheses remains bounded while the second converges to $0$. Thus,
\begin{align*}
\frac{1}{z-u_0} &= \lim_{M \to \infty} \sum_{k=0}^M \frac{1}{z-u_k} \prod_{j=k+1}^M \frac{1-u_0/u_j}{1-z/u_j}\\
&= \lim_{M \to \infty} \sum_{k=0}^M \frac{1}{z-u_k} \prod_{j=k+1}^\infty \frac{1-u_0/u_j}{1-z/u_j}\\
&= \sum_{k=0}^\infty \frac{1}{z-u_k} \prod_{j=k+1}^\infty \frac{1-u_0/u_j}{1-z/u_j},
\end{align*}
as desired.
\end{proof}

We will obtain the interpolating polynomial $p_M$ from Lemma~\ref{lemma:alg},
thereby recovering a standard contour integral representation for Hermite
interpolation. In the following calculations, we will use $u_0$ as the
variable for our interpolating polynomial, while $u_1,u_2,\dots$ will be
fixed interpolation nodes.

Let $\contour$ be the contour in the complex plane that traces the points
$x+ix$ for $x$ from $-\infty$ to $-\sqrt{2}/2$, wraps counterclockwise around
the unit circle to $-\sqrt{2}/2+i\sqrt{2}/2$, and then traces the points
$-x+ix$ for $x$ from $\sqrt{2}/2$ to $\infty$, as shown in
Figure~\ref{fig:contour}.

\begin{figure}
\begin{tikzpicture}
\draw[->,>=stealth] (-3,0) -- (3,0);
\draw[->,>=stealth] (0,-3) -- (0,3);
\draw (-0.35355339,-0.35355339) arc (-135:135:0.5);
\draw (-3,-3) -- (-0.35355339,-0.35355339);
\draw[->,>=stealth] (-0.35355339,0.35355339) -- (-3,3);
\draw (-1,1.36) node {$\contour$};
\end{tikzpicture}
\caption{The contour $\contour$, oriented counterclockwise, together with the coordinate axes.
The circular arc is part of the unit circle, and the rays are at $45^\circ$ angles from the axes.}
\label{fig:contour}
\end{figure}
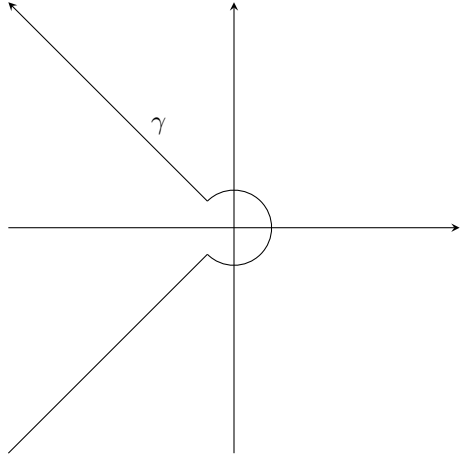

Integrating
\[
\frac{1}{2\pi i} e^{\alpha z} \frac{dz}{z-u_0}
\]
around $\contour$ and applying the first identity from Lemma~\ref{lemma:alg}
shows that for $u_0 \in (-\infty,0]$,
\begin{equation} \label{eq:contourM}
e^{\alpha u_0} = \sum_{k=0}^M H_{k,M} \prod_{j=k+1}^M\left(1-\frac{u_0}{u_j}\right),
\end{equation}
where
\[
H_{k,M} = \frac{1}{2\pi i} \int_{\contour} \frac{e^{\alpha z} \, dz}{(z-u_k)\prod_{j=k+1}^M (1-z/u_j)}.
\]
Note that the reason we use an unbounded contour $\contour$ is to avoid
having it depend on $M$; the use of $\contour$ is justified by the
exponential decay of $e^{\alpha z}$ as $\Re z \to -\infty$.

For $k \ge 1$, the coefficient $H_{k,M}$ is a constant independent of $u_0$,
while $H_{0,M}$ is a holomorphic function of $u_0$. It follows from
\eqref{eq:contourM} that
\begin{equation} \label{eq:HcoeffpM}
p_M(u_0) = \sum_{k=1}^M H_{k,M} \prod_{j=k+1}^M\left(1-\frac{u_0}{u_j}\right),
\end{equation}
because
\[
e^{\alpha u_0} - \sum_{k=1}^M H_{k,M} \prod_{j=k+1}^M\left(1-\frac{u_0}{u_j}\right) = H_{0,M} \prod_{j=1}^M\left(1-\frac{u_0}{u_j}\right),
\]
which vanishes to the desired order (i.e., the multiplicity) at each of
$u_1,\dots,u_M$, and $p_M$ is the unique polynomial of degree less than $M$
with this property.

\begin{lemma}[Cohn and Woo \cite{CW}] \label{lemma:CW}
Let $t_1,\dots,t_M$ be elements of the interval $I$, not necessarily
distinct, and let $f \colon I \to \R$ be absolutely monotonic (i.e.,
infinitely differentiable with all derivatives nonnegative). Then the Hermite
interpolation of $f$ at $t_1,\dots,t_M$ is a nonnegative linear combination
of the polynomials
\[
t \mapsto \prod_{i=1}^m (t-t_i)
\]
for $0 \le m < M$.
\end{lemma}

For a proof, see Lemma~10 in \cite{CW}. The precise interpretation of the
lemma depends on the ordering of $t_1, \dots, t_M$, but the lemma is true for
every ordering. It follows from this lemma and \eqref{eq:HcoeffpM} that
$H_{k,M} \ge 0$ for $1 \le k \le M$. Furthermore, the same argument proves
that $H_{0,M} \ge 0$ for $u_0 \in (-\infty,0]$, since $H_{0,M}$ is exactly
the same sort of coefficient for interpolation at $u_0,\dots,u_M$ rather than
$u_1,\dots,u_M$.

\section{Interpolation by entire functions of exponential type} \label{sec:entire}

So far, we have a hypothetical definition of the auxiliary function $h$ in
\eqref{eq:hdef}, but it is not clear that the limit in this equation even
exists.  To prove that it does exist and analyze its properties, we will
extend the analysis from the previous section to carry out interpolation by
entire functions of exponential type.

By the second identity in Lemma~\ref{lemma:alg},
\begin{equation} \label{eq:beforeinterchange}
e^{\alpha u_0} = \frac{1}{2\pi i} \int_{\contour} \sum_{k=0}^\infty \frac{e^{\alpha z}}{z-u_k} \prod_{j=k+1}^\infty \frac{1-u_0/u_j}{1-z/u_j} \, dz.
\end{equation}
We would like to interchange the sum and integral, but justifying this
interchange requires careful bounds. The products
\[
\prod_{j=k+1}^\infty \big(1-u_0/u_j\big)
\]
are easily bounded: if we set $u_0 = -|w|^2$, then we can rewrite these
products as
\begin{equation} \label{eq:weierstrass}
\prod_{j=k+1}^\infty \left(1 - \left(\frac{\pi r |w|}{\lambda_{\lceil j/2 \rceil}}\right)^2\right)
= c_k \frac{J_{n/2}(\pi r |w|)^2}{(\pi r |w|)^{n} \prod_{j=1}^k \Big(1-\frac{(\pi r|w|)^2}{\lambda_{\lceil j/2 \rceil}^2}\Big)}
\end{equation}
by the Weierstrass product formula for the Bessel function (see \S15.41 of
\cite{Watson}), where $c_k$ is a positive constant ensuring that the value at
$w=0$ is $1$. Notice that the Bessel function is squared because we are
interpolating to second order at each node, as reflected in
(\ref{eq:weierstrass}) by the multiplicities in the ceiling function $\lceil
j/2 \rceil$. (Second order interpolation is necessary to obtain the
inequalities that $h$ must obey.) The functions on the right of
(\ref{eq:weierstrass}) are positive definite on $\R^{n+2}$ by
Proposition~\ref{prop:peel} and the fact that positive-definite functions are
closed under multiplication.  It follows that they are bounded in absolute
value by their value when $w=0$, which is $1$.

We can also bound the effects of the $z-u_k$ factors in
\eqref{eq:beforeinterchange}. Outside of the unit circle, the points in
$\contour$ are of the form $z = x \pm ix$ with $x< 0$. At such points,
\[
|z-u_k|^2 = (x-u_k)^2+x^2 \ge u_k^2/2
\]
with equality when $x = u_k/2$, and hence
\[
\sum_{k=1}^\infty \frac{1}{|z-u_k|} \le \sqrt{2} \sum_{k=1}^\infty \frac{1}{|u_k|} < \infty.
\]

The products
\[
\prod_{j=k+1}^\infty\frac{1}{1-z/u_j}
\]
are the trickiest to handle, but we can bound them using the following lemma.

\begin{lemma} \label{lemma:upperbd}
Suppose $1 \le M \le \infty$, $0 \le k < M$, and $z = x \pm ix$ with $x<0$.
Then
\[
\prod_{j=k+1}^M \frac{1}{|1-z/u_j|^2} \le C^{\sqrt{|x|}},
\]
where $C$ is a constant that does not depend on $k$, $M$, or $x$.
\end{lemma}

\begin{proof}
We begin with
\begin{align*}
|1-z/u_j|^2 &= (1-x/u_j)^2 + (x/u_j)^2\\
&= 1 - 2x/u_j + 2x^2/u_j^2\\
&\ge 1/2,
\end{align*}
with equality when $x = u_j/2$, and the middle equation implies that
$|1-z/u_j|^2 \ge 1-2x/u_j$ for all $x$. Thus,
\[
\prod_{j=k+1}^M \frac{1}{|1-z/u_j|^2} \le 2^{O\big(\sqrt{|x|}\big)} \prod_{\substack{\text{$j$ such that}\\|u_j| > 4 |x|}} \frac{1}{1-2x/u_j},
\]
where there are only $O\big(\sqrt{|x|}\big)$ factors of $2$ because $|u_j|$
grows quadratically as a function of $j$.  Furthermore, the remaining indices
$j$ satisfy
\[
\frac{1}{1-2x/u_j} \le 1+\frac{4x}{u_j},
\]
because $(1-2\varepsilon)(1+4\varepsilon)>1$ for $0 < \varepsilon < 1/4$.  It
follows from this inequality and $1/|u_j| = O\big(1/j^2\big)$ that
\[
\prod_{j=k+1}^M \frac{1}{|1-z/u_j|^2} \le 2^{O\big(\sqrt{|x|}\big)} \prod_{j \ge 1} \left(1+\frac{K|x|}{j^2}\right) =  2^{O\big(\sqrt{|x|}\big)} \frac{\sinh \big( \pi \sqrt{K |x|}\big)}{\pi \sqrt{K |x|}}
\]
for some positive constant $K$, which completes the proof because
\[
|{\sinh v}| \le e^{|v|}/2
\]
for all $v$.
\end{proof}

Our estimates thus suffice to interchange the sum and integral in
\[
e^{\alpha u_0} = \frac{1}{2\pi i} \int_{\contour} \sum_{k=0}^\infty \frac{e^{\alpha z}}{z-u_k} \prod_{j=k+1}^\infty \frac{1-u_0/u_j}{1-z/u_j} \, dz
\]
by dominated convergence, because the exponential decay of $e^{\alpha z}$
outweighs the $C^{\sqrt{|x|}}$ growth from Lemma~\ref{lemma:upperbd}.  Thus,
\[
e^{\alpha u_0} = \sum_{k=0}^\infty H_{k,\infty} \prod_{j=k+1}^\infty\left(1-\frac{u_0}{u_j}\right),
\]
where
\begin{equation} \label{eq:Hkinfinity}
H_{k,\infty} = \frac{1}{2\pi i} \int_{\contour} \frac{e^{\alpha z} \, dz}{(z-u_k)\prod_{j=k+1}^\infty (1-z/u_j)}.
\end{equation}
Furthermore, Lemma~\ref{lemma:upperbd} suffices to show that
\[
H_{k,\infty} = \lim_{M \to \infty} H_{k,M}
\]
by dominated convergence, and hence $H_{k,\infty} \ge 0$ for all $k$.  The
integral formula \eqref{eq:Hkinfinity} for $H_{k,\infty}$ shows that it is a
constant for $k \ge 1$ and a holomorphic function of $u_0$ for $k=0$.

As above, we set $u_0 = -|w|^2$. Because
\[
p_M(u_0) = e^{\alpha u_0} - H_{0,M} \prod_{j=1}^M\left(1-\frac{u_0}{u_j}\right),
\]
the fact that $H_{0,M} \to H_{0,\infty}$ and the convergence of the infinite
product imply that $\lim_{M \to \infty} p_M(u_0)$ exists.  Thus, our
definition of $h(w)$ as this limit in \eqref{eq:hdef} yields
\begin{equation} \label{eq:hH}
h(w) = e^{\alpha u_0} -  H_{0,\infty} \prod_{j=1}^\infty\left(1-\frac{u_0}{u_j}\right) = \sum_{k=1}^\infty H_{k,\infty} \prod_{j=k+1}^\infty\left(1-\frac{u_0}{u_j}\right).
\end{equation}

We must still prove that
\begin{enumerate}
\item\label{condr:pdk} the function $h$ is continuous, integrable, and
    positive definite,

\item\label{condr:hlef} $h(x) \le f\big(|x|\big)$ for all $x \in \R^n$,

\item\label{condr:bigO} $h(x) = O\big((1+|x|)^{-(n+1)}\big)$,

\item\label{condr:supp} the support of $\widehat{h}$ is contained in
    $B^n_r(0)$, and

\item\label{condr:eq} the functions $f$ and $h$ agree at radius
    $\lambda_m/(\pi r)$ for each $m \ge 1$.
\end{enumerate}

We begin with \eqref{condr:pdk}.  The function $h$ is not just continuous,
but holomorphic. As we observed in \eqref{eq:weierstrass}, each of the
products in \eqref{eq:hH} is a positive-definite function on $\R^{n+2}$ (and
hence also when restricted to $\R^n$) by Proposition~\ref{prop:peel}, and
thus the same is true of $h$ since the cone of positive-definite functions is
closed under pointwise limits.  Integrability follows from
\[
h(w) = e^{\alpha u_0} -  H_{0,\infty} \prod_{j=1}^\infty\left(1-\frac{u_0}{u_j}\right)
\]
because the Gaussian and infinite product are integrable (for the latter see
\eqref{eq:hbigO} below) while
\[
H_{0,\infty} = \frac{1}{2\pi i} \int_{\contour} \frac{e^{\alpha z} \, dz}{(z-u_0)\prod_{j=1}^\infty (1-z/u_j)}.
\]
is bounded as a function of $u_0$ on $(-\infty,0]$ by
Lemma~\ref{lemma:upperbd}.

Furthermore, $h(w) = O\big((1+|w|)^{-(n+1)}\big)$ because
\begin{equation} \label{eq:hbigO}
\prod_{j=1}^\infty\left(1-\frac{u_0}{u_j}\right) = c_0  \frac{J_{n/2}(\pi r |w|)^2}{(\pi r |w|)^{n}} = O\big((1+|w|)^{-(n+1)}\big)
\end{equation}
by the usual asymptotics for Bessel functions (see \S7.21 of \cite{Watson}),
while $e^{-\alpha |w|^2}$ decays even faster, and thus \eqref{condr:bigO}
holds.

The equality condition \eqref{condr:eq} holds because
\[
e^{-\alpha |w|^2} - h(w) = H_{0,\infty} \prod_{j=1}^\infty\left(1-\frac{u_0}{u_j}\right).
\]
This same equation yields the inequality \eqref{condr:hlef} because
$H_{0,\infty}\ge 0$ for all $u_0 \in (-\infty,0]$ while the product is
nonnegative because each $u_j$ occurs with multiplicity two.

We must still check that the support of $\widehat{h}$ is contained in
$B_r^n(0)$.  First, note that setting $u_0=0$ in \eqref{eq:hH} shows that
\[
\sum_{k=1}^\infty H_{k,\infty} < \infty.
\]
Each summand
\[
\prod_{j=k+1}^\infty\left(1-\frac{u_0}{u_j}\right)
\]
has Fourier transform with support in $B_r^n(0)$ (after setting $u_0 =
-|w|^2$ with $w \in \R^n$), by \eqref{eq:weierstrass} and
Corollary~\ref{cor:peel}.  We conclude that $\supp(\widehat{h}) \subseteq
B_r^n(0)$ by the following lemma.

\begin{lemma}
Suppose $g_1,g_2,\dots$ are integrable functions from $\R^n$ to $\R$ such
that $|g_k| \le 1$ everywhere and $\supp(\widehat{g_k}) \subseteq B_r^n(0)$
for all $k$. Let $g = \sum_{k \ge 1} c_k g_k$, where $c_1,c_2,\dots$ satisfy
$\sum_{k \ge 1} |c_k| < \infty$, and suppose $g$ is integrable.  Then
$\supp(\widehat{g}) \subseteq B_r^n(0)$ as well.
\end{lemma}

\begin{proof}
For each $\varepsilon>0$, let $i_\varepsilon$ be a smooth, radial function
from $\R^n$ to $\R$ that is nonnegative, has integral $1$, and is supported
in $B_\varepsilon^n(0)$. Then its Fourier transform
$\widehat{\imath}_\varepsilon$ is a Schwartz function, and the infinite sum
\[
g \widehat{\imath}_\varepsilon = \sum_{k \ge 1} c_k \widehat{\imath}_\varepsilon g_k
\]
converges in $L^1$.  It follows that the sum over $k$ commutes with the
Fourier transform, and hence
\[
\widehat{g} * i_\varepsilon = \sum_{k \ge 1} c_k \widehat{g}_k
* i_\varepsilon,
\]
where $*$ denotes convolution.  The support of the right side is contained in
$B_{r+\varepsilon}^n(0)$, while the left side converges pointwise to
$\widehat{g}$ as $\varepsilon \to 0$ since $\widehat{g}$ is continuous and
convolving with $i_\varepsilon$ is an approximate identity. Thus,
$\supp(\widehat{g}) \subseteq B_r^n(0)$, as desired.
\end{proof}

We have thus shown that the auxiliary function $h$ has all the properties
\eqref{condr:pdk} through \eqref{condr:eq}, which completes the proof of
Theorem~\ref{thm:main}.

The auxiliary function $h$ is presumably not optimal in general, but it is
optimal subject to $\supp(\widehat{h}) \subseteq B^n_r(0)$ and $h(x) =
O\big((1+|x|)^{-n-\delta}\big)$ for $x \in \R$ with $\delta > 0$.  This
optimality follows immediately from Lemma~\ref{lemma:BGF}:

\begin{proposition} \label{prop:opt}
Let $f\colon \R \to \R$ be a Gaussian, let $\lambda_1 < \lambda_2 < \cdots$
be the positive roots of the Bessel function $J_{n/2}$, choose $r$ so that
$\vol\mathopen{}\big(B_{r/2}^n(0)\big)\mathclose{} = \rho$, let $h$ be an
even entire function of exponential type at most $2\pi r$, and suppose that
$h(x) \le f(x)$ and $h(x) = O\big((1+|x|)^{-n-\delta}\big)$ for $x \in \R$
with $\delta > 0$. Then
\[
\rho \widehat{h}(0)
- h(0)
\ge
\frac{n}{2^{n-1}(n/2)!^2}\sum_{m=1}^\infty \frac{\lambda_m^{n-2}}{J_{n/2-1}(\lambda_m)^2} f\mathopen{}\left(\frac{\lambda_m}{\pi r}\right)\mathclose{}.
\]
\end{proposition}

Note that we do not require $h$ to be positive definite.  This proposition is
the analogue for energy minimization of the main theorem in \cite{G} (see
also \cite{Co}), and the analogue for Euclidean space of the optimality
results in \cite{BDHSS}.

\begin{proof}
By Lemma~\ref{lemma:BGF},
\[
\rho \widehat{h}(0)
- h(0)
=
\frac{n}{2^{n-1}(n/2)!^2}\sum_{m=1}^\infty \frac{\lambda_m^{n-2}}{J_{n/2-1}(\lambda_m)^2} h\mathopen{}\left(\frac{\lambda_m}{\pi r}\right)\mathclose{},
\]
and we simply bound $h$ from above by $f$.
\end{proof}

\section{Asymptotics} \label{sec:asymp}

To prove Theorems~\ref{thm:sharp} and~\ref{thm:notsharp}, all that remains is
to carry out an asymptotic analysis to deduce them from
Theorem~\ref{thm:main}. We begin with Theorem~\ref{thm:sharp}. Consider the
Gaussian function $f(t) = e^{-\alpha t^2}$, and choose $r$ so that
$\vol\mathopen{}\big(B_{r/2}^n(0)\big)\mathclose{} = \rho$. We will assume
that $(\alpha,\rho)$ is confined to a compact subset of $(0,4\pi/e) \times
(0,\infty)$, and given such a subset all our error terms will be uniform in
$\alpha$ and $\rho$. We will estimate the sum
\begin{equation} \label{eq:normalizedsum}
\frac{n}{2^{n-1}(n/2)!^2(\pi/\alpha)^{n/2}}\sum_{m=1}^\infty
\frac{\lambda_m^{n-2}}{J_{n/2-1}(\lambda_m)^2} f\mathopen{}\left(\frac{\lambda_m}{\pi
r}\right)\mathclose{},
\end{equation}
which is the bound from Theorem~\ref{thm:main} with an extra factor of
$(\pi/\alpha)^{n/2}$ in the denominator, and we will show that it equals
$\rho+o(1)$ as $n \to \infty$.

The key is to use the uniform asymptotic formulas for Bessel function zeros
from \S10.21(viii) of \cite{NIST}.  The resulting expressions are somewhat
cumbersome, and we will not write them all explicitly, but we will specify
how to compute everything with reference to \cite{NIST}. In particular, all
references involving three numbers, such as 10.21.41, will be to equations in
\cite{NIST}.

We will find that the largest terms in the sum \eqref{eq:normalizedsum} occur
for $m$ near $cn$, where the constant $c$ depends on $\alpha$.  For $\alpha <
4\pi/e$ the constant $c$ is positive, but for $\alpha \ge 4\pi/e$ (the case
of Theorem~\ref{thm:notsharp}) it becomes $0$ and the results change
substantially.

The first step in analyzing \eqref{eq:normalizedsum} is to understand how
large $\lambda_m$ is. By 10.21.41,
\begin{equation} \label{eq:lambdam}
\lambda_m = (n/2) z_m + O(1/n),
\end{equation}
where the constant in the big-$O$ is independent of $m$, the number $z_m \ge
1$ satisfies
\begin{equation} \label{eq:zetamzm}
\frac{2}{3} (-\zeta_m)^{3/2} = \sqrt{z_m^2 - 1} - \arcsec(z_m)
\end{equation}
by 10.20.3, the number $\zeta_m < 0$ satisfies
\begin{equation} \label{eq:zetamam}
\zeta_{m} = (n/2)^{-2/3}a_m
\end{equation}
by 10.21.41, and $a_m$ is the $m$-th root of the Airy function $\Ai$, with $0
> a_1 > a_2 > \dotsb$. Note that the $O(1/n)$ error term in \eqref{eq:lambdam}
is small enough that we can estimate $\lambda_m^n$ to within a $1+o(1)$
factor.

Equations \eqref{eq:lambdam} through \eqref{eq:zetamam} reduce the problem of
estimating $\lambda_m$ to that of estimating $a_m$, and 9.9.6 tells us that
\[
a_m = -T\mathopen{}\left(\frac{3\pi (4m-1)}{8}\right)\mathclose{},
\]
where
\[
T(t) = t^{2/3}\left(1 + \frac{5}{48}t^{-2} - \frac{5}{36}t^{-4} + O\big(t^{-6}\big)\right).
\]
Using these formulas and \eqref{eq:zetamam}, we find that if $m = cn + d
\sqrt{n}$ with $c$ and $d$ bounded and $c$ bounded away from zero, then
\begin{equation} \label{eq:zetam2pic}
\frac{2}{3}(-\zeta_{m})^{3/2} = 2\pi c \left( 1 + \frac{d}{c}n^{-1/2} - \frac{1}{4c}n^{-1} + O\big(n^{-3/2}\big)\right),
\end{equation}
where the constant in the big-$O$ depends only on the bounds on $c$ and $d$.
We will choose $c$ so that the largest terms in the sum
\eqref{eq:normalizedsum} come from $m$ of this form.

We can identify the largest terms in the sum by analyzing the growth rate of
\begin{equation} \label{eq:bigfacts}
\lambda_m^n f\mathopen{}\left(\frac{\lambda_m}{\pi
r}\right)\mathclose{},
\end{equation}
because the remaining factors of $1/\big(\lambda_m^2
J_{n/2-1}(\lambda_m)^2\big)$ will turn out to have lower-order effects. If we
use \eqref{eq:lambdam} to write \eqref{eq:bigfacts} in terms of $z_m$ and use
the estimate $r \sim \sqrt{2n/(\pi e)}$ from Stirling's
formula,\footnote{Note that the estimate $\sqrt{2n/(\pi e)}$ for $r$ is not
precise enough to obtain the terms in \eqref{eq:normalizedsum} to within a
better factor than $(1+o(1))^n$. Here we use it just to identify the largest
terms in the sum, and we will use a higher-order Stirling approximation to
$r$ when computing the final answer.} then we find that
\[
\lambda_m^n f\mathopen{}\left(\frac{\lambda_m}{\pi
r}\right)\mathclose{} = (n/2)^n \left( z_m e^{- z_m^2 \alpha e/(8\pi)} +o(1) \right)^n.
\]
Thus, the exponential decay rate of the summand in \eqref{eq:normalizedsum}
is determined by $z_m e^{- z_m^2 \alpha e/(8\pi)}$, which is maximized when
$z_m = t_m$ with
\[
t_m = \sqrt{\frac{4\pi}{\alpha e}}.
\]
This tells us that the greatest contribution to the sum should come from $m$
with $z_m = (1+o(1))t_m$ as $n \to \infty$. (Note that $t_m
> 1$ because $\alpha < 4\pi/e$, in accordance with the restriction that $z_m
\ge 1$.  When $\alpha \ge 4\pi/e$ the largest terms instead come from setting
$z_m=1$.)

By \eqref{eq:zetamzm} and \eqref{eq:zetam2pic}, the value of $c$ that yields
$z_m=t_m$ in the limit as $n \to \infty$ is
\begin{equation} \label{eq:c}
c = \frac{\sqrt{\frac{4\pi}{\alpha e} - 1} - \arcsec\Big(\sqrt{\frac{4\pi}{\alpha e}}\Big)}{2\pi},
\end{equation}
and we fix that value of $c$ from this point on (while allowing $d$ to vary).
Note that because $\alpha$ is bounded away from $0$ and $4\pi/e$, $c$ is
bounded away from $\infty$ and $0$, as desired. Furthermore, one can use the
equations listed above to compute explicit constants $k_1$ and $k_2$ in terms
of $\alpha$ and $d$ such that
\[
z_m = t_m \left(1 + k_1 n^{-1/2} + k_2 n^{-1} + O\big(n^{-3/2}\big)\right)
\]
when $m = cn + d \sqrt{n}$. We omit the complicated formulas here.

Using our equations so far, we can obtain uniform asymptotics to within a
$1+o(1)$ factor for the terms
\[
\lambda_m^{n-2} f\mathopen{}\left(\frac{\lambda_m}{\pi
r}\right)\mathclose{}
\]
when $m = cn + d\sqrt{n}$ with $c$ satisfying \eqref{eq:c} and $d$ bounded.
What remains is to deal with the factor of $J_{n/2-1}(\lambda_m)^2$ in the
denominator of the summand in \eqref{eq:normalizedsum}.

It follows from 10.6.2 that
\[
J_{n/2-1}(\lambda_{m}) = J'_{n/2}(\lambda_{m}),
\]
and 10.21.42 tells us that
\begin{equation} \label{eq:Jprime}
J'_{n/2}(\lambda_m) = -\frac{2}{(n/2)^{2/3}} \frac{\Ai'(a_m)}{z_m h_m} \Big(1 + O\big(n^{-2}\big)\Big)
\end{equation}
with error term uniform in $m$, where by 10.21.45
\[
h_m = \left( \frac{4\zeta_m}{1 - z_m^2} \right)^{1/4} =  (1+o(1))\pi^{1/6}3^{1/6}2^{1/2}c^{1/6}\alpha^{1/4} \big( 4\pi/e - \alpha\big)^{-1/4}
\]
as $m \to \infty$.  To complete the analysis of \eqref{eq:Jprime}, we need an
estimate for $\Ai'(a_m)$, which follows from 9.9.7:
\[
\Ai'(a_m) = (-1)^{m-1} V\mathopen{}\left(\frac{3\pi (4m-1)}{8}\right)\mathclose{},
\]
where
\[
V(t) = \frac{t^{1/6}}{\sqrt{\pi}} \left( 1 + \frac{5}{48}t^2 - \frac{1525}{4608}t^4 + O\big(t^{-6}\big)  \right).
\]
Since $m = cn + d\sqrt{n} = (c+o(1))n$, we find that
\[
\Ai'(a_m)^{2} = (1+o(1)) \left( \frac{3cn}{2\pi^2} \right)^{1/3}.
\]

To complete the proof of Theorem~\ref{thm:sharp}, we simply combine all the
asymptotic formulas we have derived so far to estimate the terms in
\eqref{eq:normalizedsum} when $m=cn+d\sqrt{n}$ with $d$ bounded. We find that
\[
\frac{n}{2^{n-1}(n/2)!^2(\pi/\alpha)^{n/2}} \cdot
\frac{\lambda_m^{n-2}}{J_{n/2-1}(\lambda_m)^2} f\mathopen{}\left(\frac{\lambda_m}{\pi
r}\right)\mathclose{} = (\rho+o(1))e^{-\pi K d^2}\sqrt{\frac{K}{n}},
\]
where
\[
K = \frac{4\pi\alpha}{4\pi/e - \alpha}.
\]
In other words, the terms behave like a Gaussian in the variable $d$.  If we
convert the sum over $m$ with $|d| \le R$ into an integral by viewing it as a
Riemann sum, then it becomes
\begin{equation} \label{eq:Riemannsum}
(\rho+o(1)) \sqrt{K} \int_{-R}^R e^{-\pi K x^2}  \, dx
\end{equation}
in the limit as $n \to \infty$ with $R$ fixed. (Note that the factor of
$\sqrt{n}$ disappears due to the spacing in the Riemann sum.) Straightforward
error bounds show that the $o(1)$ term in \eqref{eq:Riemannsum} is uniform in
$\alpha$ and $\rho$ given our assumptions, but we do not obtain uniformity in
$R$.

All that remains is to deal with the other terms in the sum, i.e., those for
which $d$ is not bounded. Because these terms are nonnegative, we can obtain
a lower bound by simply omitting them. Thus, \eqref{eq:Riemannsum} is a lower
bound for \eqref{eq:normalizedsum} for each fixed $R$, and letting $R \to
\infty$ shows that \eqref{eq:normalizedsum} is at least $\rho+o(1)$, because
\[
\sqrt{K} \int_{-\infty}^\infty e^{-\pi K x^2}  \, dx = 1.
\]
More precisely,
\[
\sqrt{K} \int_{-R}^R e^{-\pi K x^2}  \, dx = \int_{-R\sqrt{K}}^{R\sqrt{K}} e^{-\pi y^2} \, dy,
\]
which converges uniformly to $1$ as $R \to \infty$ because $K$ is bounded
away from $0$. Combining this lower bound of $\rho+o(1)$ for
\eqref{eq:normalizedsum} with the expectation upper bound completes the proof
of Theorem~\ref{thm:sharp}.

To prove Theorem~\ref{thm:notsharp}, we simply bound the sum in
Theorem~\ref{thm:main} from below by the $m=1$ term using the same asymptotic
formulas as above.  Of course taking into account more terms would yield a
more refined estimate, but even a single term suffices to obtain the correct
exponential decay rate for the sum, and we see no need to analyze this bound
more carefully given that it is almost certainly not sharp.

\section{The conditional expectation bound}
\label{sec:condexp}

In this section we prove an upper bound for energy that we call the
\emph{conditional expectation bound}.  It refines the expectation bound by
conditioning on having no short vectors in the lattice, and it shows that the
expectation bound is not sharp for steep Gaussians, as described in
Section~\ref{sec:intro}.

\begin{lemma} \label{lem:condexp}
Let $n>1$ and $\rho > 0$, let $f \colon (0,\infty) \to \R$ be such that $x
\mapsto f\big(|x|\big)$ is integrable on $\R^n\setminus\{0\}$, and let $r\ge
0$ be such that $0 \le \vol\mathopen{}\big(B_r^n(0)\big)\mathclose{} < 2$.
Then there exists a lattice in $\R^n$ of density $\rho$ and $f${\kern
-1pt}-energy at most 
\[
\frac{\rho}{1-\vol\mathopen{}\big(B_r^n(0)\big)\mathclose{}/2}
\int_{\R^n \setminus B_r^n(0)} f\big(|x|\big) \, dx.
\]
\end{lemma}

\begin{proof}
Let $\Lambda$ be a random lattice in $\R^n$ of density $\rho$, chosen
according to the canonical probability measure on such lattices, and let $\E$
denote expectation with respect to that measure.  Then for all integrable
potential functions $g$, the $g$-energy $E_g(\Lambda)$ satisfies
\[
\E \big( E_g(\Lambda) \big) = \rho \int_{\R^n} g\big(|x|\big) \, dx
\]
by the Siegel mean value theorem.

Let $\chi_r$ be the characteristic function of $B_r^n(0)$.  We will split the
expected value
\[
\E \big( E_{f(1-\chi_r)}(\Lambda) \big) = \rho \int_{\R^n \setminus B_r^n(0)} f\big(|x|\big) \, dx
\]
into two pieces, depending on whether $\Lambda \cap B_r^n(0) = \{0\}$.  Let
$X$ denote the event that $\Lambda \cap B_r^n(0) = \{0\}$, let $\neg X$
denote the complementary event that $\Lambda \cap B_r^n(0) \ne \{0\}$, and
let $p$ be the probability of $X$. Then the conditional expectations satisfy
\[
\E \big( E_{f(1-\chi_r)}(\Lambda) \big) =
p \E \big( E_{f(1-\chi_r)}(\Lambda) \ |\  X \big)
+ (1-p) \E \big( E_{f(1-\chi_r)}(\Lambda) \ |\  \neg X\big).
\]
Conditioned on $X$, the energies with respect to $f$ and $f(1-\chi_r)$ are
identical, while the expectation conditioned on $\neg X$ is nonnegative, and
thus
\[
\E \big( E_{f(1-\chi_r)}(\Lambda) \big) \ge
p \E \big( E_{f}(\Lambda) \ |\  X \big).
\]
As long as $p > 0$, it follows that there exists a lattice $\Lambda$ with
$f${\kern -0.5pt}-energy at most 
\begin{equation} \label{eq:boundwithp}
\frac{\rho}{p} \int_{\R^n \setminus B_r^n(0)} f\big(|x|\big) \, dx,
\end{equation}
because that is an upper bound for the expected energy conditioned on $X$.

All that remains is to compute a lower bound for $p$.  The number of nonzero
lattice points in $B_r^n(0)$ is always an even integer, and by the Siegel
mean value theorem its expectation is
$\vol\mathopen{}\big(B_r^n(0)\big)\mathclose{}$. When the event $\neg X$
occurs, there are at least two nonzero lattice points in $B_r^n(0)$, and
hence $\vol\mathopen{}\big(B_r^n(0)\big)\mathclose{} \ge 2(1-p)$.  It follows
that
\[
p \ge 1-\frac{\vol\mathopen{}\big(B_r^n(0)\big)\mathclose{}}{2} > 0,
\]
and using this inequality in \eqref{eq:boundwithp} completes the proof.
\end{proof}

Computing the bound from Lemma~\ref{lem:condexp} explicitly for a Gaussian in
terms of the incomplete gamma function
\[
\Gamma(s,x) = \int_x^\infty e^{-t} t^{s} \, \frac{dt}{t}
\]
yields the following proposition, in which we take
$\vol\mathopen{}\big(B_r^n(0)\big)\mathclose{}=1/n$ for simplicity (although
this choice does not exactly optimize the bound). Note that $0 \le
\Gamma(s,x) \le \Gamma(s,0) = \Gamma(s)$ for $s
> 0$ and $x \ge 0$.

\begin{proposition} \label{prop:condexp}
Define $f \colon \R \to \R$ by $f(t) = e^{-\alpha t^2}$ with $\alpha > 0$.
Then for each $\rho>0$ and $n>1$, there exists a lattice in $\R^n$ of density
$\rho$ and $f${\kern -1pt}-energy at most 
\[
\frac{\rho}{1-1/(2n)} \left(\frac{\pi}{\alpha}\right)^{n/2} \frac{\Gamma(n/2,\alpha r^2)}{\Gamma(n/2)},
\]
where $r>0$ satisfies $\vol\mathopen{}\big(B_r^n(0)\big)\mathclose{} = 1/n$.
As $n \to \infty$ with $\alpha$ and $\rho$ fixed and $\alpha > \pi e$,
\[
\frac{\Gamma(n/2,\alpha r^2)}{\Gamma(n/2)} = \left( e^{1/2 - \alpha/(2\pi e)}+o(1) \right)^n.
\]
\end{proposition}

\begin{proof}
The bound from Lemma~\ref{lem:condexp} is
\[
\frac{\rho}{1-\vol\mathopen{}\big(B_r^n(0)\big)\mathclose{}/2} \int_{\R^n \setminus B_r^n(0)} f\big(|x|\big) \, dx
= \frac{\rho}{1-1/(2n)} \int_r^\infty e^{-\alpha u^2} n \frac{\pi^{n/2}}{(n/2)!} u^{n-1} \, du,
\]
and a change of variable from $u$ to $t = \alpha u^2$ yields the bound
\[
\frac{\rho}{1-1/(2n)} \left(\frac{\pi}{\alpha}\right)^{n/2} \frac{\Gamma(n/2,\alpha r^2)}{\Gamma(n/2)},
\]
as stated above.

To compute the asymptotics, we can use the Laplace method to estimate the
incomplete gamma function (see Chapter~4 of \cite{dB}).  Specifically, this
method shows that as $k \to \infty$,
\[
\frac{\int_{\beta +o(1)}^\infty e^{-k v} v^{k} \, dv}{\int_0^\infty e^{-k v} v^{k} \, dv} =
\begin{cases}
1+o(1) & \textup{if $\beta < 1$, and}\\
e^{k(1-\beta+o(1))} & \textup{if $\beta > 1$.}
\end{cases}
\]
The reason is that the primary contribution to the integral in the
denominator comes from near $v = 1$ and occurs on a scale of $1/\sqrt{k}$.
When $\beta < 1$, the same contribution occurs in the numerator, while the
primary contribution to the numerator comes from the left endpoint when
$\beta > 1$.

To apply these asymptotics, we write
\[
\frac{\Gamma(n/2,\alpha r^2)}{\Gamma(n/2)} =
\frac{\int_{\alpha r^2}^\infty e^{-t} t^{n/2-1} \, dt}{\int_{0}^\infty e^{-t} t^{n/2-1} \, dt}
= \frac{\int_{\alpha r^2/(n/2-1)}^\infty e^{-(n/2-1) v} v^{n/2-1} \, dv}{\int_0^\infty e^{-(n/2-1) v} v^{n/2-1} \, dv}.
\]
The equation $\vol\mathopen{}\big(B_r^n(0)\big)\mathclose{} = 1/n$ yields
$r^2 \sim n/(2\pi e)$ as $n \to \infty$, and hence $\alpha r^2/(n/2-1) \to
\alpha/(\pi e)$. Thus, for $\alpha < \pi e$,
\[
\frac{\Gamma(n/2,\alpha r^2)}{\Gamma(n/2)} = 1+o(1)
\]
as $n \to \infty$, while for $\alpha > \pi e$,
\[
\frac{\Gamma(n/2,\alpha r^2)}{\Gamma(n/2)} =
\left( e^{1/2 - \alpha/(2\pi e)}+o(1) \right)^n,
\]
as desired.
\end{proof}

The factor of $1-\vol\mathopen{}\big(B_r^n(0)\big)\mathclose{}/2$ in
Lemma~\ref{lem:condexp} plays an essential role in our proof, but we suspect
that it may not be needed for the lemma to hold.  One reason is that by using
the methods of \cite{Vance,Ve} to average over lattices invariant under
multiplication by the $k$-th roots of unity, we can replace this factor with
$1-\vol\mathopen{}\big(B_r^n(0)\big)\mathclose{}/k$ whenever $n$ is a
multiple of $\varphi(k)$ with $n > \varphi(k)$ and $0 \le
\vol\mathopen{}\big(B_r^n(0)\big)\mathclose{} < k$; here $\varphi$ is the
Euler totient function, defined by $\varphi(k) = \#\{i \in \{1,2,\dots,k\} :
\gcd(i,k)=1\}$. Thus, for certain dimensions we can achieve a factor much
closer to $1$, and we imagine that the same might hold for every dimension.

Another motivation is the following lemma:

\begin{lemma} \label{lem:dual}
Let $f \colon (0,\infty) \to \R$ be such that $x \mapsto f\big(|x|\big)$ is
integrable on $\R^n\setminus\{0\}$, let $h \colon \R^n \to \R$ be continuous,
positive definite, and integrable and satisfy $h(x) \le f\big(|x|\big)$ for
all $x \in \R^n \setminus \{0\}$, and let $\rho > 0$ and $r>0$ be such that
$\vol\mathopen{}\big(B_r^n(0)\big)\mathclose{} = 1/\rho$.  Then
\[
\rho \widehat{h}(0) - h(0) \le \rho \int_{\R^n \setminus B_r^n(0)} f\big(|x|\big) \, dx.
\]
\end{lemma}

Recall that $\rho \widehat{h}(0) - h(0)$ is the lower bound for energy in
Proposition~\ref{prop:LP}.  Thus, Lemma~\ref{lem:dual} says that the linear
programming bound behaves as if there were a point configuration of energy
\[
\rho \int_{\R^n \setminus B_r^n(0)} f\big(|x|\big) \, dx
\]
when $\vol\mathopen{}\big(B_r^n(0)\big)\mathclose{} = 1/\rho$, regardless of
whether such a configuration in fact exists. For comparison, when $\rho >
1/2$ Lemma~\ref{lem:condexp} guarantees the existence of a configuration of
energy
\[
\frac{\rho}{1-1/(2\rho)}
\int_{\R^n \setminus B_r^n(0)} f\big(|x|\big) \, dx,
\]
which is slightly worse, and when $\rho \le 1/2$ we cannot even apply
Lemma~\ref{lem:condexp} with $\vol\mathopen{}\big(B_r^n(0)\big)\mathclose{} =
1/\rho$.

\begin{proof}
As above, let $\chi_r$ be the characteristic function of $B_r^n(0)$. Then
\begin{align*}
\rho \int_{\R^n \setminus B_r^n(0)} f\big(|x|\big) \, dx &= \rho \int_{\R^n} (1-\chi_r)f\\
&\ge \rho \int_{\R^n} (1-\chi_r)h\\
&= \rho \widehat{h}(0) - \rho \int_{\R^n} \chi_r h\\
&= \rho \widehat{h}(0) - \rho \int_{\R^n} \widehat{\chi}_r \widehat{h},
\end{align*}
where the last equality follows from the Plancherel theorem.  Note that $h$
is in $L^2(\R^n)$ because it is integrable and bounded.

Because $\widehat{\chi}_r$ is positive definite, it is never larger than
$\widehat{\chi}_r(0) = \vol\mathopen{}\big(B_r^n(0)\big)\mathclose{} =
1/\rho$. Furthermore, $\widehat{h} \ge 0$ and $\widehat{h}$ is integrable
(see, for example, Lemma~5.6.2 in \cite{R}). Thus,
\begin{align*}
\rho \widehat{h}(0) - \rho \int_{\R^n} \widehat{\chi}_r \widehat{h} &\ge \rho \widehat{h}(0) -  \int_{\R^n} \widehat{h}\\
&= \rho \widehat{h}(0) - h(0),
\end{align*}
as desired.
\end{proof}

In terms of tempered distributions, Lemma~\ref{lem:dual} amounts to using the
distribution $\delta_0 + \rho - \rho \chi_r$ as a feasible point in the dual
to the linear programming bound, but we have expressed it directly in terms
of elementary manipulations so that we can apply it to auxiliary functions
that are not necessarily Schwartz functions.  Our choice of distribution is
analogous to those used in \cite{TS} and \cite{SST}, but it is slightly
simpler in that those papers also include a spherical delta function at
radius $r$.  One could likely prove a better bound by taking this extra term
into account, but we have opted for simplicity.

\section{Sphere packing and the Gaussian core model} \label{sec:spherepacking}

In this section, we describe a more precise connection between sphere packing
and the Gaussian core model.  It will be convenient to normalize the sphere
packing problem in $\R^n$ as follows: what is the smallest $r>0$ for which
there is a discrete subset $\mC$ of $\R^n$ with density~$1$ and minimal
distance at least $r\sqrt{n}$ (i.e., no two distinct points in $\mC$ are
separated by a distance of less than $r \sqrt{n}$)?  Centering spheres of
radius $r\sqrt{n}/2$ at the points of such a configuration yields a packing
density of
\[
\vol\mathopen{}\big(B_{r\sqrt{n}/2}^n(0)\big)\mathclose{}
= \frac{\pi^{n/2}}{(n/2)!} \left( \frac{r\sqrt{n}}{2} \right)^n
= \left(\frac{r\sqrt{2\pi e}}{2}+o(1)\right)^n
\]
as $n \to \infty$. The best sphere packings known in high dimensions achieve
packing density $(1/2+o(1))^n$, i.e., $r = 1/\sqrt{2\pi e} + o(1)$, and
determining whether this value of $r$ is optimal is a major open problem.

The following proposition gives an upper bound for the Gaussian energy of a
sphere packing when the Gaussian is steep enough. The primary difficulty is
ruling out large numbers of pairs of points just slightly further apart than
the minimal distance, and a crude bound based on volume suffices.

\begin{proposition} \label{prop:packingtoenergy}
Fix $\alpha>0$ and $r>0$, and let $f(t) = e^{-\alpha t^2}$. If $2 \alpha r^2
\ge 1$, then every discrete subset of $\R^n$ with density $1$ and minimal
distance at least $r \sqrt{n}$ has
$f${\kern -0.5pt}-energy 
at most
\[
\big(3 e^{-\alpha r^2} + o(1)\big)^n
\]
as $n \to \infty$.
\end{proposition}

\begin{proof}
Let $\mC$ be such a configuration.
To estimate the $f${\kern -0.5pt}-energy 
of $\mC$, we will show that
\[
\sum_{y \in \mC\setminus\{x\}} e^{-\alpha|x-y|^2} \le \big(3 e^{-\alpha r^2} + o(1)\big)^n
\]
uniformly for each $x \in \mC$.

To bound this sum, we let $\varepsilon>0$ and look at concentric open balls
$U_0, U_1, \dots$ about $x$, where
\[
U_i = \{y \in \R^n: |y-x| < (1+\varepsilon)^i r \sqrt{n}\}.
\]
By hypothesis, there are no points of $\mC$ in $U_0$ except $x$.  Thus,
\begin{align*}
\sum_{y \in \mC\setminus\{x\}} e^{-\alpha|x-y|^2} &= \sum_{k=0}^\infty \sum_{y \in \mC \cap (U_{k+1} \setminus U_k)} e^{-\alpha|x-y|^2}\\
&\le \sum_{k=0}^\infty |\mC \cap (U_{k+1} \setminus U_k)| \, 
e^{-\alpha r^2 (1+\varepsilon)^{2k} n}\\
&\le \sum_{k=0}^\infty |\mC \cap U_{k+1}| \, 
e^{-\alpha r^2 (1+\varepsilon)^{2k} n}.
\end{align*}

To bound $|\mC \cap U_{k+1}|$, note that if we place a sphere of radius
$r\sqrt{n}/2$ at each point of $\mC \cap U_{k+1}$, then these spheres do not
overlap and all lie within a sphere of radius $(1+\varepsilon)^{k+1} r
\sqrt{n} + r\sqrt{n}/2$.  Thus,
\begin{equation} \label{eq:volume-bound}
|\mC \cap U_{k+1}| \le \frac{\vol \mathopen{}\big(B^n_{(1+\varepsilon)^{k+1} r
\sqrt{n} + r\sqrt{n}/2}\big)\mathclose{}}{\vol \mathopen{}\big(B^n_{r\sqrt{n}/2}\big)\mathclose{}} = \big(2 (1+\varepsilon)^{k+1} + 1\big)^n,
\end{equation}
and our upper bound becomes
\[
\sum_{k=0}^\infty  \big(2 (1+\varepsilon)^{k+1} + 1\big)^n
e^{-\alpha r^2 (1+\varepsilon)^{2k} n}.
\]

Because $2 (1+\varepsilon)^{k+1} + 1 \le 3 (1+\varepsilon)^{k+1}$ and
$-\alpha r^2 (1+\varepsilon)^{2k} n < -\alpha r^2 (1+2k\varepsilon) n$, we
obtain a simpler upper bound of
\begin{align*}
\sum_{k=0}^\infty  \big(3 (1+\varepsilon)^{k+1}\big)^n
e^{-\alpha r^2 (1+2k\varepsilon) n} &=
3^n e^{-\alpha r^2 n} (1+\varepsilon)^n \sum_{k=0}^\infty  \big((1+\varepsilon)^{n} e^{-2\alpha r^2 n \varepsilon} \big)^k\\
&= \frac{3^n e^{-\alpha r^2 n} (1+\varepsilon)^n}{1 - (1+\varepsilon)^n e^{-2\alpha r^2 n \varepsilon}},
\end{align*}
where the geometric series converges because our hypothesis that $2 \alpha
r^2\ge 1$ implies that $(1+\varepsilon) e^{-2\alpha r^2 \varepsilon} < 1$.

Thus, we obtain an energy upper bound of a constant times $3^n e^{-\alpha r^2
n} (1+\varepsilon)^n$ for each fixed $\varepsilon$ as $n \to \infty$, and
letting $\varepsilon \to 0$ shows that the energy is at most $\big(3
e^{-\alpha r^2} + o(1)\big)^n$.
\end{proof}

The factor of $e^{-\alpha r^2 n}$ in this bound is the dominant factor when
$\alpha$ is large, and it is essentially optimal: given an upper bound of
$e^{-\alpha s^2 n}$ for energy as $n \to \infty$ for some constant $s$, a
$1-o(1)$ fraction of points must have no neighbors at distance less than
$(s+o(1))\sqrt{n}$ (or else the energy would be too large), and we can then
obtain minimal distance $(s+o(1))\sqrt{n}$ by removing a negligible fraction
of the points.  In other words, every low-energy configuration can be
modified to form a dense packing, which is a partial converse to
Proposition~\ref{prop:packingtoenergy}.

By contrast, the factor of $3^n$ comes from the volume estimate
\eqref{eq:volume-bound}.  One could improve this factor by using a more
sophisticated packing density bound in place of \eqref{eq:volume-bound}, but
that would not suffice to eliminate it entirely. For comparison, the energy
must be at least $K e^{-\alpha r^2 n}$, where $K$ is the average kissing
number (i.e., the average number of neighbors at distance $r \sqrt{n}$), and
this factor of $K$ might grow exponentially with $n$.

If we apply Proposition~\ref{prop:packingtoenergy} to the best packings
currently known, with $r = 1/\sqrt{2\pi e} + o(1)$, then we obtain an upper
bound of $\big(3e^{- \alpha/(2\pi e)}+o(1)\big)^n$ when $\alpha \ge \pi e$,
compared with $\left(e^{-\alpha/(2\pi e)}\sqrt{\pi e/\alpha} + o(1)
\right)^n$ from Proposition~\ref{prop:condexp}. The dominant behavior as
$\alpha \to \infty$ is the same, with the difference being the lower-order
factors of $3$ or $\sqrt{\pi e/\alpha}$. In particular, the coefficient of
$-1/(2\pi e)$ for $\alpha$ in Proposition~\ref{prop:condexp} cannot be
improved unless there exist exponentially denser sphere packings than those
currently known.

Another consequence of Proposition~\ref{prop:packingtoenergy} is that lower
bounds for energy yield upper bounds for the sphere packing density.  Suppose
there exist packings with $\rho=1$ and $r = \beta/\sqrt{2\pi e} + o(1)$ as $n
\to \infty$, i.e., packing density $(\beta/2 + o(1))^n$.  Then
Proposition~\ref{prop:packingtoenergy} gives an energy upper bound of $\big(3
e^{-\alpha \beta^2/(2\pi e)} + o(1)\big)^n$ when $\alpha$ is large enough,
while Theorem~\ref{thm:notsharp} gives a lower bound of
\[
\left( \frac{1}{2} e^{1-\alpha e/(8\pi)} + o(1) \right)^n.
\]
Comparing these bounds yields
\[
\frac{1}{2} e^{1-\alpha e/(8\pi)} \le 3
e^{-\alpha \beta^2/(2\pi e)},
\]
and as $\alpha \to \infty$ we find that
\[
-\frac{e}{8\pi} \le -\frac{\beta^2}{2\pi e}.
\]
In other words, $\beta \le e/2$, which means no sphere packing can have
packing density greater than $(e/4+o(1))^n$.  This argument recovers
Levenshtein's bound \cite{L}, which is no surprise given that similar Bessel
function constructions prove this bound directly \cite{Co,CE03,G}.
Levenshtein's bound is not the best upper bound known for the sphere packing
density in high dimensions. The best bound known can be obtained from the
linear programming bound \cite{KL,CZ14}, and it is natural to hope for a
corresponding lower bound for energy, but we do not have one.

\section{Open problems}
\label{sec:open}

Many problems remain open, most notably determining the true asymptotics for
minimal energy and for the linear programming bound.  We expect that both the
upper and the lower bounds can be improved when $\alpha$ is large, and we
have no idea which may be closer to the truth.

One natural extension of our work would be to consider other potential
functions, particularly inverse power laws.  As discussed in
Section~\ref{sec:intro}, the conditional expectation bound is sharp to within
a constant factor for inverse power laws $t \mapsto 1/t^{n+s}$ with $s>0$
fixed as $n \to \infty$, and that constant factor would become $1+o(1)$ if
Theorem~\ref{thm:sharp} could be extended to all $\alpha < \pi e$.  We
conjecture that it can be extended:

\begin{conjecture} \label{conj:sharp}
When $f(t) = e^{-\alpha t^2}$ with $0 < \alpha < \pi e$, the minimal
$f${\kern -1pt}-energy in $\R^n$ for configurations of density $\rho$ is 
$(\rho+o(1))(\pi/\alpha)^{n/2}$ as $n \to \infty$ with $\alpha$ and $\rho$
fixed, or more generally with $(\alpha,\rho)$ confined to a compact subset of
$(0,\pi e) \times (0,\infty)$.
\end{conjecture}

The conditional expectation bound shows that the bound $\pi e$ for $\alpha$
cannot be increased in Conjecture~\ref{conj:sharp}.  What happens beyond that
point?  We would guess that there is some range of $\alpha$ over which the
conditional expectation bound is asymptotically sharp.  Perhaps it is sharp
for all $\alpha > \pi e$, or perhaps there are further phase transitions when
$\alpha$ is large, after which other bounds take over.

Note that the transition at $\pi e$ affects only the asymptotics for energy,
and not the form of the ground states.  Specifically, for fixed $\alpha>\pi
e$ and $\rho$, a random lattice in $\R^n$ of density $\rho$ has energy at
most
\[
{\rho}(\pi/\alpha)^{n/2}\left( e^{1/2 - \alpha/(2\pi e)}+o(1) \right)^n
\]
with probability $1-o(1)$ as $n \to \infty$, because the bound in
Proposition~\ref{prop:condexp} is the expectation conditioned on an event of
probability $1-o(1)$.  (If there were a more than $o(1)$ chance of greater
energy, then that alone would increase the expectation.) Thus, we have not
disproved the hypothesis that random lattices asymptotically minimize
Gaussian energy in high dimensions for all $\alpha$, as they do for $\alpha <
4\pi/e$. We suspect that this hypothesis is false, but we cannot propose any
better constructions.  As $\alpha$ grows, we expect that these problems will
become increasingly difficult to resolve.

\end{document}